\documentclass[a4paper,11pt,UKenglish]{amsart}

\usepackage[utf8]{inputenc}
\usepackage[T1]{fontenc}

\usepackage{babel}

\usepackage{amsmath}
\usepackage{amssymb}
\usepackage{amsthm}
\usepackage{stmaryrd}
\usepackage{thmtools}
\usepackage{mathtools}
\usepackage{hyperref}
\usepackage[capitalize]{cleveref}
\usepackage{enumitem}
\usepackage{subcaption}
\usepackage{todonotes}
\usepackage{xspace}

\usepackage{tikz}
\usetikzlibrary{arrows.meta}
\usetikzlibrary{decorations.markings}
\tikzstyle vertex=[circle, draw, fill=black, inner sep=0mm, minimum width=1.5mm]
\tikzstyle fakevertex=[circle, draw, fill=none, inner sep=0mm, minimum width=1.5mm]
\tikzstyle smallvertex=[circle, draw, fill=black, inner sep=0mm, minimum width=0.9mm]
\tikzstyle directed=[postaction={decorate,decoration={markings, mark=at position .65 with {\arrow{Straight Barb}}}}]
\tikzstyle reverse directed=[postaction={decorate,decoration={markings, mark=at position .65 with {\arrowreversed{Straight Barb};}}}]
\tikzstyle arc=[-{>[length=2mm, width=2mm]}]
\tikzstyle colormap=[dotted,semithick,-{>[length=1.5mm,width=1.5mm]}]

\crefname{claim}{Claim}{Claims}

\Crefname{figure}{Figure}{Figures}

\crefformat{equation}{\textup{#2(#1)#3}}
\crefrangeformat{equation}{\textup{#3(#1)#4--#5(#2)#6}}
\crefmultiformat{equation}{\textup{#2(#1)#3}}{ and \textup{#2(#1)#3}}
{, \textup{#2(#1)#3}}{, and \textup{#2(#1)#3}}
\crefrangemultiformat{equation}{\textup{#3(#1)#4--#5(#2)#6}}%
{ and \textup{#3(#1)#4--#5(#2)#6}}{, \textup{#3(#1)#4--#5(#2)#6}}{, and \textup{#3(#1)#4--#5(#2)#6}}

\Crefformat{equation}{Equation~\textup{#2(#1)#3}}
\Crefrangeformat{equation}{Equations~\textup{#3(#1)#4--#5(#2)#6}}
\Crefmultiformat{equation}{Equations~\textup{#2(#1)#3}}{ and \textup{#2(#1)#3}}
{, \textup{#2(#1)#3}}{, and \textup{#2(#1)#3}}
\Crefrangemultiformat{equation}{Equations~\textup{#3(#1)#4--#5(#2)#6}}%
{ and \textup{#3(#1)#4--#5(#2)#6}}{, \textup{#3(#1)#4--#5(#2)#6}}{, and \textup{#3(#1)#4--#5(#2)#6}}

\crefname{section}{section}{sections}

\newtheorem{theorem}{Theorem}[section]
\newtheorem{proposition}[theorem]{Proposition}

\newtheorem{lemma}[theorem]{Lemma}
\theoremstyle{definition}
\newtheorem{definition}[theorem]{Definition}
\theoremstyle{remark}
\newtheorem{remark}[theorem]{Remark}
\newtheorem{claim}[theorem]{Claim}

\newenvironment{claimproof}[1][Proof.]{\begin{proof}[#1]}{\end{proof}}

\newcommand{\card}[1]{|{#1}|}
\newcommand{\Nn}{\mathbb{N}}

\newcommand{\C}{\mathcal{C}}

\newcommand{\Zz}{\mathbb{Z}}
\newcommand{\Cc}{\mathcal{C}}

\newcommand{\Pc}{\mathcal{P}}
\newcommand{\Tc}{\mathcal{T}}

\newcommand{\eqdef}{\coloneqq}
\newcommand{\setst}[2]{\textstyle \left\{#1 \ : \ #2 \right\}}

\newcommand{\qgt}{\succ}
\newcommand{\qle}{\preceq}
\newcommand{\qge}{\succeq}

\newcommand{\from}{\leftarrow}

\newcommand{\ie}{i.e.\@\xspace}

\renewcommand{\le}{\leqslant}
\renewcommand{\ge}{\geqslant}
\renewcommand{\emptyset}{\varnothing}
\renewcommand{\phi}{\varphi}

\DeclareMathOperator{\rk}{\sf rk}

\DeclareMathOperator{\nlcw}{\sf nlcw}
\DeclareMathOperator{\rw}{\sf rw}

\DeclareMathOperator{\lnlcw}{\sf lnlcw}
\DeclareMathOperator{\lrw}{\sf lrw}
\DeclareMathOperator{\type}{\sf type}
\DeclareMathOperator{\idx}{\sf idx}
\DeclareMathOperator{\arity}{\sf ar}

\newcommand{\fo}{\textsf{FO}\xspace}
\newcommand{\mso}{\textsf{MSO}\xspace}
\newcommand{\emso}{\textsf{EMSO}\xspace}
\newcommand{\cmso}{\textsf{CMSO}\xspace}
\newcommand{\colinit}{\lambda_{init}}
\newcommand{\colfin}{\lambda_{fin}}
\newcommand{\edgeint}{E^{\text{int}}}
\newcommand{\edgebnd}{E^{\partial}}
\newcommand{\abstr}[1]{\llbracket #1 \rrbracket}

\title{Transducing Linear Decompositions of Tournaments}
\author{Colin Geniet}
\author{Fatemeh Ghasemi}
\author{Mamadou Moustapha Kanté}

\address{Institute for Basic Science (IBS), Discrete Mathematics Group, Daejeon, South Korea}
\email{research@colingeniet.com}

\address{LACL, Université Paris-Est Créteil, France}
\email{fatemeh.ghasemi@lacl.fr}

\address{Université Clermont Auvergne, Clermont Auvergne INP, LIMOS, CNRS, Clermont-Ferrand, France}
\email{mamadou.kante@uca.fr}

\thanks{C.~Geniet was supported by the Institute for Basic Science (\nolinkurl{IBS-R029-C1}).
  F.~Ghasemi is supported by the ANR project DIFFERENCE (\href{https://anr.fr/Projet-ANR-20-CE48-0002}{\nolinkurl{ANR-20-CE48-0002}}) and by IUT Sénart-Fontainebleau.
}

\begin{document}
\begin{abstract}
  Bojańczyk, Pilipczuk, and Grohe [LICS '18] proved that for graphs of bounded linear clique-width,
  clique-decompositions of bounded width can be produced by a \textsf{CMSO} transduction.
  We show that in the case of tournaments, a first-order transduction suffices.
  This implies that the logics \textsf{CMSO} and \emph{existential} \textsf{MSO} are equivalent over bounded linear clique-width tournaments.
\end{abstract}
\maketitle

\section{Introduction}
Bojańczyk and Pilipczuk proved the following conjecture of Courcelle:
\begin{theorem}[\cite{bojanczyk2016definability,BojanczykP17}]
  \label{thm:treewidth}
  For any~$k$, there is a monadic second-order (\mso) transduction~$\Phi$ from graphs to tree-decompositions of graphs such that, on every input graph $G$, the
  following holds:

  \begin{enumerate}
  \item  If $G$ is a graph of tree-width at most $k$, then $\Phi$ non-deterministically outputs a tree-decomposition of $G$ of width at most $k$, and
  \item every output of $\Phi$ is a tree-decomposition of width at most $k$.
  \end{enumerate}
\end{theorem}
This can be understood as saying that tree-decomposition can be constructed
not in a usual algorithmic sense, but through a process described by \mso logic.
The motivation for this result is to prove the converse of Courcelle's theorem:
for graphs of tree-width~$k$, if~$\Pc$ is a property that can be tested by a tree automaton running on tree decompositions of width~$k$, then~$\Pc$ is definable by an \mso formula.

The proof of \cref{thm:treewidth} proceeds in two main steps:
\begin{enumerate}
  \item \label{item:pathwidth} First, one considers bounded \emph{path-width} graphs.
    Path decompositions are seen as words in a semigroup,
    to which can be applied the Factorisation Forest Theorem of Simon~\cite{simon1990factorization}.
    This reduces the problem to path decomposition with a crucial regularity property---idempotence---%
    which notably ensures the existence of paths without excessive detours.
  \item \label{item:treewidth} In the general case, one constructs
    a tree decomposition whose bags have bounded path-width, and are in a sense maximal.
    Careful connectivity considerations, involving path systems with low congestion,
    allow to MSO transduce this decomposition, reducing the problem to the previous path-width case.
\end{enumerate}

A natural generalisation of \cref{thm:treewidth} is to consider dense graphs,
replacing tree-width by clique-width.
The same authors and Grohe generalised step~\eqref{item:pathwidth} to this setting,
replacing path-width by its dense equivalent \emph{linear clique-width},
and using the slightly stronger logic \cmso, i.e.\ \mso with counting modulo~2.
Their proof is once again crucially based on Simon's theorem.
\begin{theorem}[\cite{bojanczyk2018cliquewidth}]\label{thm:linear-cw} For any~$k$, there is a \cmso transduction $\Phi$ from graphs to clique-decompositions, and~$k' \in \Nn$ such that
  for any input graph~$G$,
  \begin{enumerate}
  \item If $G$ has linear clique-width at most $k$, then $\Phi$ non-deterministically outputs a clique-decomposition of $G$ of width at most $k'$,
  \item every output of $\Phi$ is a clique-decomposition of width at most $k'$.
  \end{enumerate}
\end{theorem}
Pushing this generalisation to graphs of bounded clique-width however remains an open problem:
the second step of the proof of \cref{thm:treewidth} seems strongly tied to the behaviour of paths in tree decompositions,
hence much harder to generalise to dense graphs.
The proof of \cref{thm:linear-cw} is already quite involved,
requiring to define a dense equivalent of connectivity.

We propose considering these questions for a restricted class of graphs,
namely \emph{tournaments}, i.e.\ directed graphs in which between every pair of vertices~$x,y$,
exactly one of the edges~$x \to y$ or~$y \to x$ exists.
In this setting, we strengthen \cref{thm:linear-cw} in two ways:
\begin{itemize}
  \item We only use first-order (\fo) logic, which is much weaker than \mso.
  \item Our transduction always produces \emph{linear} decompositions,
    whereas the output decompositions sometimes need to be tree-like
    in \cref{thm:linear-cw} and step~\eqref{item:pathwidth} of \cref{thm:treewidth}.
\end{itemize}
That is, we prove the following.
\begin{restatable}{theorem}{transducedecomp} \label{thm:main}
  For any~$k$, one can compute an \fo transduction~$\Phi$ from graphs to linear clique-decompositions, and $k' \in \Nn$, such that on any input graph~$T$,
  \begin{enumerate}
    \item if $T$ is a tournament of linear clique-width at most $k$, then $\Phi$ non-deterministically outputs a linear clique-decomposition of~$T$ of width at most $k'$, and
  \item every output of $\Phi$ is a linear clique-decomposition of~$T$ of width at most~$k'$.
  \end{enumerate}
\end{restatable}
The proof of \cref{thm:main} is also significantly simpler than \cref{thm:linear-cw}.
The key difference is that we do not refer to any notion of connectivity.
Oversimplifying greatly, in the path-width case,
the relative ordering of vertices in the path decomposition
is determined by queries such as `is there a small separator between~$x$ and~$y$'.
For tournaments, our queries are rather `is~$x \to y$ an edge'
or `is there a vertex~$z$ with edges~$x \to z \to y$'.
The latter, unlike the former, can be expressed using first-order logic.

In a sense, this is possible because in tournaments, the highly regular substructures provided by Simon's factorisation are linear orders (also called transitive tournaments);
by contrast, in the graph case, these highly regular substructures could be paths or edgeless graphs, for which \fo is insufficient to describe a decomposition.
Simon's factorisation theorem remains a key tool in our proof.

\Cref{thm:linear-cw}, together with Courcelle's Backwards Translation Theorem \cite[Theorem~1.40]{CourcelleE2012},
implies that for bounded linear clique-width graphs,
a property~$\Pc$ is \cmso definable if and only if it is testable by a finite automaton running on linear clique decompositions.
In the case of tournaments, our \fo transduction implies that the same holds with
\emph{existential \mso} (\emso) instead of \cmso.
Therefore, by going back and forth through automata on linear clique decompositions, we obtain the following.
\begin{restatable}{theorem}{emsoequiv} \label{thm:emso}
  For any~$k$ and any \cmso sentence~$\phi$ on tournaments,
  there is an \emso formula~$\psi$ such that~$\phi$ and~$\psi$
  are equivalent for tournaments of linear clique-width at most~$k$.
\end{restatable}

Let us conclude by discussing potential generalisations of our results to (non-linear) clique-width, and limits to them.
\Cref{thm:main} cannot directly generalise to bounded clique-width tournaments:
A Ramsey-like argument of Mikołaj Bojańczyk (personal communication) shows that
for~$\Tc$ the class of tournaments obtained by iterated lexicographic products of directed triangles
(which has clique-width~3),
no \fo transduction can produce a clique decomposition of bounded width for all tournaments in~$\Tc$.
On the other hand, it is simple to do so for the class~$\Tc$ with a transduction in the logic \emph{\fo with counting modulo~2} (\textsf{FO$+$C}).
To our knowledge, it is possible that \cref{thm:main} generalises to clique-width if we allow \textsf{FO$+$C} transductions.

A major part of our proof focuses on constructing a vertex ordering with low cut-rank, as a first step towards a linear clique decomposition.
We believe that asking to transduce \emph{any} total vertex ordering is an interesting intermediate question.
For which tournament classes~$\Cc$ and logics~$\mathcal{L}$ is there an $\mathcal{L}$-transduction that produces a total vertex ordering on any given $T \in \Cc$?
Bojańczyk's example once again implies that this is not possible for bounded clique-width tournaments and the logic~$\fo$, but \textsf{FO$+$C} could be sufficient.
Note that from a clique decomposition of small width of a tournament (represented with the ancestor--descendent relation of the tree),
one can always transduce a linear vertex ordering, corresponding to some left-to-right ordering on the leaves of the decomposition tree.\footnote{
  In the context of e.g.\ \cref{thm:linear-cw},
  one needs to consider a generalisation of clique decompositions
  where the decomposition tree may have nodes of unbounded degree.
  These high degree nodes make it impossible to transduce a total vertex ordering from the decomposition tree.
  In the case of tournaments however, they cannot exist.
  Indeed, a high degree node implies that there is a large subset~$X$ of vertices such that
  every bipartition of~$X$ is a cut of bounded rank, which is impossible in a tournament.
}
Thus transducing a total ordering is an easier question than transducing clique-decompositions.

\subsection*{Structure of the paper}
\Cref{Section: Preliminaries} presents standard definitions and results on clique-width and \fo transductions.
In \cref{Section: Linear decompositions and semigroups}, we introduce an infinite monoid $\mathfrak{C}_K$ corresponding to linear clique decompositions of width~$K$, and a finite quotient~$M_K$ of~$\mathfrak{C}_K$.
This gives us the power to use \emph{Simon's forest factorisation theorem}, stated in \cref{subsection: simon},
which for any linear clique decomposition in $\mathfrak{C}_K$ yields a \emph{factorisation}, subject to regularity conditions relative to~$M_K$, and with depth only depending on~$|M_K|$.
In \cref{sec:main}, we use induction on the factorisation given by Simon's theorem to prove a weakening of \cref{thm:main}: we construct a transduction which yields not a clique decomposition, but simply the associated vertex ordering.
\Cref{sec:definable-decomp} then shows how to transduce from this ordering to the clique decomposition, completing the proof of \cref{thm:main}, and finally shows how it implies \cref{thm:emso}.

\section{Preliminaries}\label{Section: Preliminaries}
We denote by~$[n]$ the interval of integers~$\{1,\dots,n\}$.

A \emph{tournament} $T = (V,E)$ is a directed graph with for each pair $u \neq v$ of vertices in~$V$, exactly one of the two possible directed edges~$(u,v)$ or~$(v,u)$ in~$E$.
We also write $V(T) \eqdef V$ and $E(T) \eqdef E$ for the vertex and edge sets of~$T$,
and denote an oriented edge $(u,v) \in E(T)$ as $u \to v$ to emphasise its orientation.
A \emph{bipartite tournament} $B = (L,R,E)$ is defined similarly restricting to edges between the two vertex sets~$L$ and~$R$: the edge set~$E$ consists of
exactly one of $u \to v$ or $v \to u$ for each pair of vertices~$u \in L$ and~$v \in R$.

In a tournament~$T$, two disjoint subsets of vertices $A,B \subset V(T)$ are called \emph{homogenous} if the edges between them are either all oriented
from~$A$ to~$B$, \ie $a \to b$ is an edge for all~$a \in A$ and $b \in B$, or inversely all from~$B$ to~$A$.

\subsection{Linear NLC-width and rank-width}
There are several graph complexity measures that are equivalent to (linear) clique-width. For our purposes, we will use two of them instead of linear
clique-width, namely \emph{linear NLC-width} and \emph{linear rank-width} (see for instance \cite{GurskiR19} for definitions).

For the definition of linear NLC-width, we restrict ourselves to tournaments, and refer to \cite{GurskiR19} for the definition on all graphs.
NLC-width and clique-width (and their linear variants) are very similar in definitions, and differ by a factor of at most~2, see for instance
\cite{gurski2005relationship,johansson1998clique,johansson2001labels} for the proof of the equivalences in the case of undirected graphs and \cite{GurskiWY16}
for directed versions.

A \emph{$k$-labelled tournament} is a tuple $T=(V,E,\lambda)$ where $(V,E)$ is a tournament equipped with a labelling map $\lambda : V \to [k]$.
We also often call~$\lambda(v)$ the \emph{colour} of the vertex~$v$; however~$\lambda$ is not a proper colouring.
\begin{definition}[Linear NLC-width]
  The class $\Tc_{k}$ consists of the $k$-labelled tournaments obtained by iteratively applying the following two operations, starting from the empty tournament:
  \begin{description}
    \item[vertex addition] Given a $k$-labelled tournament $T = (V,E,\lambda)$, a subset~$X \subseteq [k]$, and a colour~$c \in [k]$,
      add a new vertex~$v$ to~$V$, add for each~$u \in V$ the edge $u \to v$ when~$\lambda(u) \in X$, and~$u \from v$ otherwise, and finally extend~$\lambda$ to~$v$ as $\lambda(v) = c$.
    \item[relabelling] Given a $k$-labelled tournament $T = (V,E,\lambda)$ and a function $\rho : [k] \to [k]$, replace the labelling~$\lambda$ by $\rho \circ \lambda$.
  \end{description}
  The least~$k$ such that $T \in \Tc_k$ is called the \emph{linear NLC-width} of $T$, denoted by $\lnlcw(T)$.

  A sequence of vertex-addition and relabelling operations witnessing that $\lnlcw(T) \le k$ is called a \emph{linear NLC decomposition of width~$k$} of~$T$, or simply linear decomposition.
  Given a linear decomposition~$D$ of~$T$, the order in which vertices are added in~$D$ defines a linear ordering~$<$ of~$V(T)$.
  When~$D$ is a linear decomposition of width~$k$, we say that~$<$ is a \emph{linear decomposition ordering of width~$k$}.
  By extension, we define the \emph{relative NLC-width} of a linear ordering~$<$ of~$V$ as the smallest~$k$ such that~$<$ is a linear decomposition ordering of width~$k$.
\end{definition}

Let us now define \emph{linear rank-width},
which will be more convenient to bound the width of the output of our transductions.
We again restrict the definition to tournaments, but refer to for instance~\cite{AdlerFP14,oum2017rank} for more information on (linear) rank-width of
undirected graphs and to \cite{GurskiR19,KanteR13} for its relations to other parameters of directed graphs. Since we restrict ourselves to tournaments, the definition below is slightly
different from the one in  \cite{GurskiR19,KanteR13} where the field on $4$ elements is used while we use here the binary field, but the two definitions are
equivalent on tournaments. 
The adjacency matrix of a tournament~$T$ is the $\mathbb{F}_2$-matrix whose rows and columns are indexed by $V(T)$, where the entry at row~$x$ and column~$y$ is~$1$ only if $(x,y)\in E(T)$.

\begin{definition}[Linear rank-width]
  In a tournament~$T$, consider a bi-partition $V(T) = X \uplus Y$ of the vertices.
  The \emph{rank} of this bi-partition, denoted by~$\rk_T(X;Y)$, is the rank (over the binary field~$\mathbb{F}_2$) of the adjacency matrix of~$X$ versus~$Y$, \ie the
  adjacency matrix restricted to rows indexed by vertices in $X$ and columns indexed by vertices in $Y$.
  If $x_1 < \dots < x_n$ is a linear ordering of~$V(T)$, the \emph{cut-rank} of~$<$ is defined as
  \[ \max_i \rk\left(\{x_1,\dots,x_i\}; \{x_{i+1},\dots,x_n\}\right). \]
  The \emph{linear rank-width} of~$T$, denoted by $\lrw(T)$, is the minimum cut-rank over all linear orderings of~$V(T)$.
\end{definition}

Rank-width and clique-width (or NLC-width) are famously equivalent up to an exponential bound~\cite{oum2006approximating}.
The proof in fact relates the different widths of any fixed ordering.
Since we need this slightly more precise result, we state it below and provide a proof in the case of
tournaments, however a similar statement holds for all graphs.
\begin{lemma}
  \label{lem:cw-rw}
  For any linear ordering~$<$ of a tournament~$T$, if~$\nlcw$ and $\rw$ denote, respectively, the relative NLC-width and cut-rank of~$<$ in~$T$, then
  \[ \rw \le \nlcw \le 2^{\rw}+1. \]
\end{lemma}

\begin{proof}
  Enumerate~$V(T)$ as $v_1 < \dots < v_n$.

  First assume that~$<$ is a linear NLC decomposition ordering of width~$k$.
  In the corresponding linear decomposition, call $\lambda_i : \{v_1,\dots,v_i\} \to [k]$ the labelling obtained just after the addition of~$v_i$.
  Then for any $j \le i$ and $\ell > i$, the direction of the edge $v_j \to v_\ell$ or $v_j \from v_\ell$ only depends on~$v_\ell$ and the label~$\lambda_i(v_j)$.
  It follows that the adjacency matrix of $\{v_1,\dots,v_i\}$ versus $\{v_{i+1},\dots,v_n\}$ has at most~$k$ distinct rows corresponding to the~$k$ possible labels in~$\lambda_i$, hence its rank is at most~$k$.
  This proves that~$<$ has cut-rank at most~$k$.

  Conversely, assume that~$<$ has cut-rank at most~$k$.
  For each~$i \in [n]$, define the partition~$\Pc_i$ of~$\{v_1,\dots,v_i\}$ into neighbourhood classes:
  two vertices are in the same class when they have exactly the same neighbours in~$\{v_{i+1},\dots,v_n\}$.
  The assumption that $\rk\left(\{v_1,\dots,v_i\}; \{v_{i+1},\dots,v_n\}\right) \le k$ implies that~$\Pc_i$ contains at most~$2^k$ parts.
  For each~$i$, choose a labelling map $\lambda_i : \{v_1,\dots,v_i\} \to [2^k]$ such that $\Pc_i = \setst{\lambda_i^{-1}(c)}{c \in [2^k]}$.
  From the definition, it is clear that each part of~$\Pc_i$ is contained within a part of~$\Pc_{i+1}$.
  This implies that the restriction of~$\lambda_{i+1}$ to~$\{v_1,\dots,v_i\}$ can be written as $\rho_i \circ \lambda_i$ for some $\rho_i : [2^k] \to [2^k]$.

  Suppose we have already constructed a linear NLC-decomposition for $T[\{v_1,\dots,v_i\}]$ of width~$2^k+1$ with labelling~$\lambda_i$ and insertion ordering~$<$.
  Note that the colour~$2^k+1$ is unused in~$\lambda_i$.
  For any $j \le i$, the direction of the edge between~$v_j$ and~$v_{i+1}$
  only depends on the part of~$\Pc_i$ containing~$v_j$, hence only on the label~$\lambda_i(v_j)$.
  Thus adding~$v_{i+1}$ with the desired edges is a valid vertex addition operation.
  We assign the last colour~$2^k+1$ to~$v_{i+1}$ at this step.
  Then, to obtain the labelling~$\lambda_{i+1}$ we apply the relabelling map~$\rho_i$, extended by mapping~$2^k+1$ to the desired colour for~$v_{i+1}$.
  Repeating these two steps for each~$i$ yields a linear NLC decomposition of width~$2^k+1$, where vertices are added in the order given by~$<$,
  thus proving that~$<$ is a linear NLC decomposition ordering of width at most~$2^k+1$.
\end{proof}

The following lemma will be used to bound the cut-rank of the ordering obtained in \cref{thm:main}.
\begin{lemma}
  \label{lem:lexico-rankwidth} Let $<$ be a linear ordering of the vertices of a tournament $T$ and let
  $X_1 < \dots < X_\ell$  be a partition of $V(T)$ into intervals of $<$ such that, for each $1\leq i\leq \ell$, 
  \begin{enumerate}
    \item $\rk(X_1 \cup \dots X_i; X_{i+1} \cup \dots X_\ell) \le k$, and
    \item the linear ordering $<$ restricted to the sub-tournament $T[X_i]$ has cut-rank at most~$k'$.
  \end{enumerate}
  Then~$<$ has cut-rank at most~$2k+k'$.
\end{lemma}
\begin{proof}
  Consider a bi-partition $V(T) = Y \uplus Z$ with~$Y$ and~$Z$ prefix and suffix respectively of the ordering~$<$.
  Call~$X_i$ the last interval intersected by~$Y$, and define
  \begin{align*}
    Y' & \eqdef Y \cap X_i &  Y'' & \eqdef X_1 \cup \dots \cup X_{i-1}, \quad \text{and symmetrically}  \\
    Z' & \eqdef Z \cap X_i  & Z'' & \eqdef X_{i+1} \cup \dots \cup X_\ell.
  \end{align*}
  Note that $Y = Y' \cup Y''$ and $Z = Z' \cup Z''$.
  By assumption, $\rk(Y'' \cup X_i; Z'')$ and $\rk(Y''; Z'' \cup X_i)$ are at most~$k$.
  Also, $Y',Z'$ is a prefix--suffix bi-partition of~$X_i$, hence $\rk(Y';Z') \le k'$.
  Thus
  \begin{align*}
    \rk(Y;Z) & \le \rk(Y';Z') + \rk(Y;Z'') + \rk(Y'';Z) \\
    & \le \rk(Y';Z') + \rk(Y'' \cup X_i;Z'') + \rk(Y''; Z \cup X_i) \\
    & \le k' + 2k. \qedhere
  \end{align*}
\end{proof}

\subsection{First-order logic, interpretations, and transductions}
\label{subsection: interpretation and transductions}
A \emph{relational signature}~$\Gamma$ is a finite set of \emph{relation symbols}~$R \in \Gamma$ each with an \emph{arity}~$\arity(R) \in \Nn$.
A \emph{relational structure}~$S$ over the signature~$\Gamma$, or $\Gamma$-structure for short,
consists of a \emph{universe} or \emph{vertex set}~$V(S)$, and for each symbol~$R \in \Gamma$ with arity~$r = \arity(R)$, a valuation of the symbol as a relation $R^S \subseteq (V(S))^r$.
For instance, directed simple graphs with loops are exactly $\{E\}$-structures, where~$E$ is a binary symbol (meaning~$\arity(E) = 2$) representing directed
edges. The transduction in \cref{thm:main} will take as inputs $\{E\}$-structures.

\fo formulas over the signature~$\Gamma$ are formulas which can be evaluated in a $\Gamma$-structure.
They are defined inductively as follows:
\begin{description}
  \item[atomic formulas] For each symbol~$R \in \Gamma$ with arity~$r$, there is a predicate~$R(x_1,\dots,x_r)$
    which checks that the vertices represented by the variables~$x_1,\dots,x_r$ form an $r$-tuple in the relation~$R$.
    There is also an equality predicate~$x = y$, checking that the variables~$x,y$ represent the same vertex.
  \item[combinators] \fo formulas $\phi,\psi$ can be combined with the usual boolean operations $\neg \phi$, $\phi \wedge \psi$, $\phi\vee\psi$.
  \item[quantifiers] \fo formulas allow quantification over the vertices of the structure $\exists x,\,\phi$ and $\forall x,\,\phi$, where~$\phi$ is itself an $\fo$ formula and~$x$ is a variable.
\end{description}
When a formula~$\phi$ has free variables~$x_1,\dots,x_r$ (i.e.\ variables appearing without any binding quantifier), we write them as parameters~$\phi(x_1,\dots,x_r)$.
Given a $\Gamma$-structure~$S$ and vertices $v_1,\dots,v_r \in V(S)$, one can define whether $\phi(x_1,\dots,x_r)$ is satisfied by~$S$ when evaluating~$x_i$ as~$v_i$ in the obvious way;
this is denoted as $S \models \phi(v_1,\dots,v_r)$.

Let us now define \emph{\fo transductions}. We follow \cite{bojanczyk2016definability} for our terminology. An \emph{\fo transduction} from a signature~$\Gamma$ to a signature~$\Delta$ is a non-deterministic (one-to-many) map from $\Gamma$-structures to $\Delta$-structures, consisting of any sequence of the following operations.
\begin{description}
  \item[interpretation]
    An interpretation~$\Phi$ from signature~$\Gamma$ to~$\Delta$ is described by a formula~$\phi_R(x_1,\dots,x_r)$ over~$\Gamma$ for each symbol~$R \in \Delta$ with arity~$r$.
    It deterministically maps a $\Gamma$-structure~$S$ to the $\Delta$-structure~$\Phi(S)$
    with same vertices and with relations
    \[ R^{\Phi(S)} = \setst{(v_1,\dots,v_r)}{S \models \phi_R(v_1,\dots,v_r)}. \]
  \item[colouring] This operation enables non-determinism by extending the signature~$\Gamma$ with unary relations (\ie with arity~1).
    For unary relation symbols~$C_1,\dots,C_c$ not present in~$\Gamma$,
    colouring a $\Gamma$-structure~$S$ by $C_1,\dots,C_c$ outputs the set of all $(\Gamma \cup \{C_1,\dots,C_c\})$-structures~$S^+$
    which coincide with~$S$ on~$\Gamma$ (meaning $V(S^+) = V(S)$ and $R^{S^+} = R^S$ for all~$R \in \Gamma$).
  \item[filtering] The colouring operations outputs all possible colourings, but we may only be interested in some of them.
    The filtering operation allows to remove unwanted structures.
    For a sentence~$\phi$ over~$\Gamma$, given a $\Gamma$-structure~$S$, the operation of filtering by~$\phi$ outputs~$S$ itself when~$S \models \phi$, and outputs nothing otherwise.
  \item[universe restriction]
    The restriction operation is described by a single unary formula~$\phi(x)$, and deterministically maps $\Gamma$-structures to $\Gamma$-structures by deleting any vertex~$v$ which does not satisfy~$\phi(v)$.
  \item[$k$-copying]
    Given a $\Gamma$-structure~$S$, this operation outputs the structure consisting of~$k$ disjoint copies of~$S$,
    together with a new relation~$C$ of arity~$k$ interpreted as the set of tuples $(v_1,\dots,v_k)$ for each~$v \in V(S)$ where~$v_i$ denotes the $i$th copy of~$v$.
\end{description}

Each of the previous operations~$f$ has input and output signatures~$\Gamma$ and~$\Delta$,
and maps a $\Gamma$-structure~$S$ to a set $f(S)$ of $\Delta$-structures.
For deterministic operations (interpretations, universe restriction, copying), $f(S)$ is a singleton,
for colouring it is a non-empty set, and for filtering it is either~$\{S\}$ or~$\emptyset$.
When~$\C$ is a set of~$\Gamma$-structures, we write $f(\C) = \bigcup_{S \in \C} f(S)$.
Thus, these operations map sets of $\Gamma$-structures to sets of $\Delta$-structures, allowing to compose them when the signatures match. 

It is well known that any \fo transduction can be equivalently expressed using only one instance of each operation, in the following order:
copying, colouring, filtering, interpretation, and universe restriction.
This follows from the Backwards Translation Theorem \cite{CourcelleE2012}, which roughly says that, for any interpretation~$\Phi$ from~$\Gamma$ to~$\Delta$ and formula~$\psi$ over~$\Delta$,
there is a formula $\psi \circ \Phi$ over~$\Gamma$ such that $S \models \psi \circ \Phi$ if and only if $\Phi(S) \models \psi$,
obtained from~$\psi$ by replacing a predicate~$R \in \Delta$ by the formula~$\phi_R$ defining it in~$\Phi$.

An \fo transduction is said to be \emph{non-copying} if it does not use the copying operation, and is called \emph{deterministic} if it does not use colouring.
Given signatures $\Gamma \subseteq \Delta$, a $\Gamma$-structure~$S$ and a $\Delta$-structure~$T$,
we say that~$S$ is a \emph{reduct} of~$T$, and~$T$ an \emph{extension} of~$S$, if $V(S) = V(T)$, and $R^S = R^T$ for all symbols $R \in \Gamma$.
That is, $T$ is exactly the structure~$S$ with some additional relations added for the symbols in $\Delta \setminus \Gamma$.
A transduction $\Phi : \Gamma \to \Delta$ is an \emph{extension transduction}
if for any $\Gamma$-structure~$S$, any $T \in \Phi(S)$ is an extension of~$S$.

Let us finally recall the Parallel Application Lemma, which allows to apply in parallel the same transduction to arbitrarily many disjoint substructures of a given structure.
Let~$\Gamma$ be a signature and~$S_1,\dots,S_n$ be disjoint $\Gamma$-structures.
Define the \emph{disjoint union} of $S_1,\dots,S_n$, denoted by~$\bigsqcup_{i=1}^n S_i$,
as the following structure over the vocabulary~$\Gamma\cup\{\thicksim\}$
where~$\thicksim$ is a new binary relation symbol:
\begin{itemize}
  \item the universe of~$\bigsqcup_{i=1}^n S_i$
    is the disjoint union of the universes of the $S_i$s;
  \item for each relation symbol from~$\Gamma$,
    its interpretation in~$\bigsqcup_{i=1}^n S_i$
    is the union of the interpretations in each of the~$S_i$s;
  \item the interpretation of~$\thicksim$ in~$\bigsqcup_{i=1}^n S_i$
    is the set of pairs of elements that originate from the same~$S_i$.
\end{itemize}
 
\begin{lemma}[Parallel Application Lemma \cite{bojanczyk2018cliquewidth}]
   \label{lem:parallel-transduction}
   For any \fo-transduction $\tau : \Sigma \to \Gamma$, there is a second \fo-transduction $\hat{\tau} : (\Sigma\cup\{\thicksim\}) \to (\Gamma\cup\{\thicksim\})$ such that
   for any input sequence~${I}_1,\dots,{I}_n$ of $\Sigma$-structures and $\mathcal{I} \eqdef \bigsqcup_{i=1}^n I_i$,
   the outputs $\mathcal{O} \in \hat{\tau}(\mathcal{I})$ are exactly the structures of the form
   $\mathcal{O} = \bigsqcup_{i=1}^n {O}_i$ with $O_i \in \tau(I_i)$ for all $i \in [n]$.
 \end{lemma}

\subsection{Example: transducing cut-width orderings}
To illustrate the previous notions and introduce some conventions, let us prove a very simple case of \cref{thm:main}.
For a tournament~$T$ and an ordering $v_1,v_2,\dots,v_n$ of the vertices, a \emph{cut} of this ordering is a partition into two intervals $\{v_1,\dots,v_i\}$ and $\{v_{i+1},\dots,v_n\}$.
The \emph{width} of this cut is the number of edges oriented backwards $u \from v$ where $u \in \{v_1,\dots,v_i\}$ and $v \in \{v_{i+1},\dots,v_n\}$.
The \emph{cut-width} of this ordering is the maximum width of its cuts. The \emph{cut-width} of~$T$ is the minimum cut-width over all linear orderings of its vertices.

For any~$k$, we want a transduction~$\Phi$ which, given a tournament~$T$, produces orderings of cut-width~$k$, if any.
The signature of~$\Phi$ should be $\{E\} \to \{E,<\}$ (the symbols~$E,<$ being both binary):
the transduction is given a tournament~$(V,E)$, and outputs an \emph{ordered tournament}~$(V,E,<)$, where~$<$ should be the cut-width ordering.
Since we expect the vertex and edge sets to remain the same in the output, we also ask~$\Phi$ to be an \emph{extension transduction}.

\begin{proposition}\label{prop:cutwidth}
  For any~$k$, one can compute an extension \fo-transduction $\Phi : \{E\} \to \{E,<\}$ such that for any tournament $T = (V,E)$,
\[ \Phi(T) = \{(V,E,<) : \text{$<$ is a cut-width~$k$ ordering of~$T$}\}. \]
\end{proposition}
\begin{proof}
  The transduction~$\Phi$ is the composition of two steps:
  first an extension transduction~$\tau : \{E\} \to \{E,<\}$ which produces \emph{at least} the desired set of ordered tournaments,
  i.e.\ $(V,E,<) \in \tau(V,E)$ for any cut-width~$k$ ordering~$<$,
  followed by a filtering transduction to remove any undesired structure.

  Let us start with the filtering step.
  Over the signature~$\{E,<\}$, with~$E$ supposedly being the edge relation of a tournament and~$<$ a linear ordering, there are simple \fo formulas (depending on~$k$) expressing the following conditions:
  \begin{itemize}
    \item the relation~$<$ is interpreted as a linear ordering, and
    \item for each vertex~$v$, there are at most~$k$ distinct edges $x_i \from y_i$ in~$E$ with $x_i \le v < y_i$,
      i.e.\ the cut $\{w: w \le v\},\{w: w>v\}$ has width at most~$k$.
  \end{itemize}
  The filtering transduction simply checks that both are satisfied.

  Let us now turn our attention to the extension transduction that will output all linear orderings of cut-width at most $k$.
  Let~$T$ be a tournament and $v_1 < \dots < v_n$ a linear ordering of cut-width at most~$k$ of $T$.
  First, if there are~$2k+1$ vertices~$z_1,\dots,z_{2k+1}$, each satisfying $x \to z_i \to y$, then $x<y$ must hold.
  Indeed, at most~$k$ out-neighbours of~$x$ can be before~$x$ in~$<$,
  and symmetrically at most~$k$ in-neighbours of~$y$ are before~$y$, leaving at least one vertex satisfying $x < z_i < y$.
  This criteria is expressed in \fo as
  \[\phi(x,y) \eqdef \exists z_1,z_2,\dots, z_{2k+1}, \ \bigwedge_{i=1}^{2k+1} E(x,z_i)\wedge E(z_i,y).\]

We cannot expect~$\phi(x,y)$ by itself to describe a total ordering of~$T$, but it is not too far from doing so: $\phi(v_i,v_j)$ holds whenever~$j-i > 4k+1$.
Indeed, of the~$4k+1$ vertices $v_{i+1},\dots,v_{j-1}$, at most~$k$ can be in-neighbours of~$v_i$, and at most~$k$ out-neighbours of~$v_j$.
This leaves~$2k+1$ of them satisfying $v_i \to v_\ell \to v_j$, hence~$\phi(v_i,v_j)$ holds.

The previous two facts imply that two cut-width~$k$ orderings of the same tournament cannot be very different: they agree on all but~$(4k+1)n$ of the pairs of vertices.
There remains some choice in the ordering of the remaining pairs, which is why we need non-determinism.

We now begin the construction of our transduction itself.
The first step is to non-deterministically colour~$T$ with~$8k+5$ colours.
We assume that this colouring assigns to~$v_i$ the colour~$i$ modulo~$8k+5$, and show that this specific choice gives~$<$ as output.
For this specific colouring, if vertices~$v_i,v_j$ are at distance at most~$4k+2$, one can immediately tell whether $v_i < v_j$ or $v_j < v_i$ simply by looking at their colours. This can be encoded by a simple, but tedious \fo formula~$\gamma(x,y)$.
We then apply the following interpretation to obtain~$<$: given vertices~$u,v$,
\begin{enumerate}
  \item if~$\phi(u,v)$ holds, then~$u < v$ and symmetrically if~$\phi(v,u)$ holds, then~$u > v$, and
  \item otherwise, $u,v$ must be at distance at most~$4k+2$ (in the successor relation of $<$), and we test whether~$u < v$ by using their colours.
\end{enumerate}
This description of~$<$ is encoded by the following formula:
\[ \psi_<(x,y) = \phi(x,y) \lor (\lnot \phi(y,x) \land \gamma(x,y)). \]
Since we wish to preserve the edge relation~$E$ (in order to check that each linear ordering has cut-width at most $k$), we add a second `identity' formula $\psi_E(x,y) = E(x,y)$.
The formulas~$\psi_E,\psi_<$ define an interpretation from signature $\{E,C_1,\dots,C_{8k+5}\}$ to $\{E,<\}$.

In the end, the transduction consists of first colouring with~$C_1,\dots,C_{8k+5}$, then interpreting by~$(\psi_E,\psi_<)$, and finally checking (with the filtering operation) that~$<$ indeed has cut-width at most~$k$.
\end{proof}

\section{Linear decompositions and semigroup theory}\label{Section: Linear decompositions and semigroups}
This section presents linear decompositions as constructed from pieces called \emph{bags},
which can be composed, giving them a monoid structure.
In this context, we will describe a linear decomposition as a sequence of small bags.
We then introduce Simon's Factorisation Forest Theorem,
which will be used to obtain factorisations of such sequences, subject to some regularity conditions.

While all the notions presented here can be defined for directed graphs of small NLC-width, we restrict the definitions
to tournaments to simplify notions and notations. We refer to \cite{bojanczyk2018cliquewidth} for a similar monoid for undirected graphs of small clique-width, which trivially extends to directed graphs.

\subsection{Linear decompositions as a monoid}\label{Section: Linear decompositions as a monoid}
A \emph{bag} is essentially a tournament with some additional information describing how it should be glued with other bags, so as to define a deterministic product operation on bags.
Precisely, for~$k \in \Nn$, a bag~$B$ of order~$k$ consists of the following:
\begin{enumerate}
 \item An \emph{internal tournament} $(V(B),\edgeint(B))$,
   whose vertices~$V(B)$ and edges~$\edgeint(B)$ are called \emph{internal vertices} and \emph{internal edges}.
 \item A \emph{boundary bipartite tournament} $([k], V(B), \edgebnd(B))$.
   Here, $1,\dots,k$ are called \emph{input vertices}, and stand for colour classes of vertices to the left of~$B$.
 \item A \emph{colouring}~$\lambda_B : V(B) \to [k]$ of the internal vertices.
 \item A \emph{recolouring} function~$\rho_B : [k] \to [k]$, which can be seen as a re-colouring of the input vertices.
 \end{enumerate}

 A visual representation of bags is depicted in \cref{fig:bag}.

 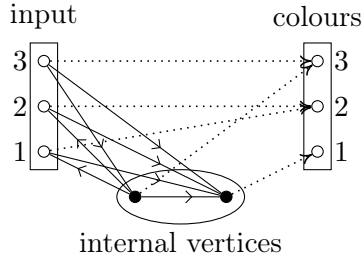
\begin{figure}[htb]
  \begin{center}
    \begin{tikzpicture}[scale=1.2]
      \foreach \i in {1,2,3}{
        \node[fakevertex] (i\i) [label=left:$\i$] at (0,0.5*\i) {};
        \node[fakevertex] (o\i) [label=right:$\i$] at (3,0.5*\i) {};
      }

      \node[vertex] (x) at (1,0) {};
      \node[vertex] (y) at (2,0) {};

      \draw (1.5,0) ellipse (0.7cm and 0.3cm);
      \node at (1.5,-0.5) {internal vertices};

      \draw (-0.15,0.3) rectangle (0.15,1.7);
      \node at (0,2) {input};

      \draw (2.85,0.3) rectangle (3.15,1.7);
      \node at (3,2) {colours};

      \foreach \s/\t in {i3/o3,i2/o2,i1/o2,x/o3,y/o1}{
        \draw[colormap] (\s) -- (\t);
      }
      \foreach \s/\t in {x/i1,x/i2,i3/x,i1/y,i2/y,i3/y,x/y}{
        \draw[directed] (\s) -- (\t);
      }
    \end{tikzpicture}
  \end{center}
  \caption{%
    Representation of a bag~$B$.
    Input vertices and colours are drawn as empty nodes on the left and right respectively,
    while the internal vertices are filled.
    The colouring maps~$\lambda_B$ and~$\rho_B$ are drawn as dotted arrows,
    while edges in~$\edgeint(B)$ and~$\edgebnd(B)$ are solid arrows.
  }
  \label{fig:bag}
\end{figure}

 The objective is to prove that if a tournament has linear NLC-width $k$, then it is a value of a word from a finite alphabet composed of bags. Let us define for
that the multiplication operator on bags.  Intuitively, when multiplying on the right by the bag~$B$, the following happens: edges are added between any~$v \in V(B)$ and all
existing vertices of colour~$c$ according to the direction of the edge~$c \to v$ or~$c \from v$ in~$\edgebnd(B)$.  Then, existing vertices have their colours
modified by applying~$\rho_B$, while internal vertices of~$B$ are given their colour from~$\lambda_B$.

Formally, the product $B_1 \cdot B_2$ of two bags $B_1$ and $B_2$ of order~$k$ is defined as follows.

\begin{enumerate}
  \item The internal vertices are~$V(B_1\cdot B_2) \eqdef V(B_1) \uplus V(B_2)$.
    Internal edges~$\edgeint(B_1\cdot B_2)$ are obtained as follows:
    Inside~$V(B_1)$ and~$V(B_2)$ respectively, edges are exactly as in~$\edgeint(B_1)$ and~$\edgeint(B_2)$.
    Between them, for~$x \in V(B_1)$ with colour $c \eqdef \lambda_{B_1}(x)$ and~$y \in V(B_2)$,
    there is an edge~$x \to y$ (resp.~$x \from y$) if and only if
    $c \to y$ (resp.~$c \from y$) is a boundary edge in $\edgebnd(B_2)$.
  \item For~$i \in [k]$ an input vertex and~$x \in V(B_2)$,
    there is a boundary edge~$i \to x$ (resp. $x\to i$) in $\edgebnd(B_1\cdot B_2)$ if and only if there is the edge $\rho_{B_1}(i) \to x$ (resp. $x\to \rho_{B_1}(i)$) in $\edgebnd(B_2)$.
    When~$x \in V(B_1)$, the edge between~$i$ and~$x$ is simply the same as in~$\edgebnd(B_1)$.
  \item The colouring~$\lambda_{B_1\cdot B_2}$ of internal vertices coincides
    with $\rho_{B_2} \circ \lambda_{B_1}$ inside~$V(B_1)$, and simply with~$\lambda_{B_2}$ inside~$V(B_2)$.
  \item Finally, the recolouring function is~$\rho_{B_2} \circ \rho_{B_1}$.
\end{enumerate}
In this product, we think of~$B_1$ as being \emph{to the left of}, or \emph{earlier} than~$B_2$.
See \cref{fig:product} for an example.
\begin{figure}[htb]
  \begin{center}
    \begin{tikzpicture}
      \foreach \i in {1,2}{
        \node[fakevertex] (i\i) [label=left:$\i$] at (0,0.7*\i) {};
        \node[fakevertex] (m\i) at (2.5,0.7*\i) {};
        \node[fakevertex] (o\i) [label=right:$\i$] at (5,0.7*\i) {};
      }
      \node[vertex] (x) [label=below right:$x$] at (1.25,0.3) {};
      \node[vertex] (y) [label=below right:$y$] at (1.25,-0.7) {};
      \node[vertex] (z) [label=below:$z$] at (3.75,0) {};

      \draw (-0.15,0.5) rectangle (0.15,1.6);
      \draw (2.35,0.5) rectangle (2.65,1.6);
      \draw (4.85,0.5) rectangle (5.15,1.6);

      \foreach \s/\t in {i2/m1,i1/m1,x/m2,y/m1,m1/o2,m2/o1,z/o1}{
        \draw[colormap] (\s) -- (\t);
      }
      \foreach \s/\t in {x/i1,i2/x,i1/y,i2/y,x/y,z/m1,m2/z}{
        \draw[directed] (\s) -- (\t);
      }

      \node at (6.25,1.05) {$=$};

      \begin{scope}[xshift=7.5cm]
      \foreach \i in {1,2}{
        \node[fakevertex] (i\i) [label=left:$\i$] at (0,0.7*\i) {};
        \node[fakevertex] (o\i) [label=right:$\i$] at (4,0.7*\i) {};
      }
      \node[vertex] (x) [label=below right:$x$] at (1.2,0) {};
      \node[vertex] (y) [label=below right:$y$] at (1.2,-1) {};
      \node[vertex] (z) [label=below right:$z$] at (2.8,0) {};

      \draw (-0.15,0.5) rectangle (0.15,1.6);
      \draw (3.85,0.5) rectangle (4.15,1.6);

      \foreach \s/\t in {i1/o2,i2/o2,x/o1,y/o2,z/o1}{
        \draw[dotted,semithick,-{>[length=1.5mm,width=1.5mm]}] (\s) -- (\t);
      }
      \foreach \s/\t in {x/i1,i2/x,i1/y,i2/y,z/i1,z/i2,x/y,x/z,z/y}{
        \draw[directed] (\s) -- (\t);
      }
      \end{scope}
    \end{tikzpicture}
  \end{center}
  \caption{%
    Product of two bags. The two bags are drawn on the left,
    with the output vertices of the first identified with the input of the second.
    The result is on the right.
  }
  \label{fig:product}
\end{figure}
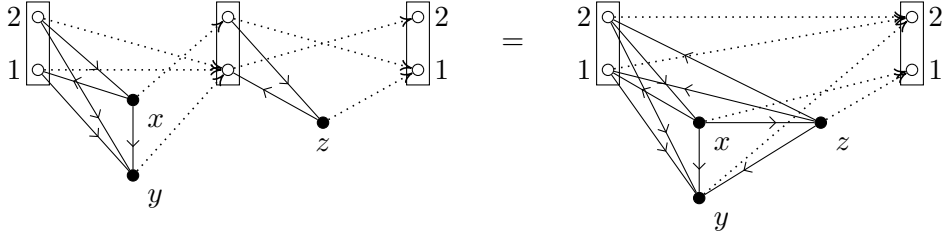

Note that we only define the product of two bags of the same order,
although it would be simple to lift this restriction by extending the set of colours of one bag.
It is easy to check that the product is associative,
and the empty bag with identity recolouring is the neutral element.
Thus, the set of bags of order~$k$ has a monoid structure. We often write $AB$ instead of~$A\cdot B$ when it is clear from the context.

A bag is called \emph{atomic} if it has at most one internal vertex.  We denote by~$\mathfrak{C}_k$ the submonoid generated by atomic bags of order~$k$.

\begin{lemma}\label{lem:monoid-nlcwidth} If a tournament $T$ has linear NLC-width at most $k$, then $T$ is
  isomorphic to the internal tournament of some $B\in \mathfrak{C}_k$. Conversely, if $T$ is isomorphic to the internal tournament of some $B\in
  \mathfrak{C}_k$, then $T$ has linear NLC-width at most $k+1$.
\end{lemma}
\begin{proof}
  In a linear NLC decomposition, a vertex-addition operation with parameters $X \subseteq [k]$ and $c \in [k]$ is represented by the atomic bag $B(X,c)$
  with a single internal vertex~$v$, boundary edges $v \to i$ if and only if $i\in X$, identity recolouring function, and~$c$ as colour for~$v$.
  The operation of recolouring by the function $\rho : [k] \to [k]$ is represented by the atomic bag $B(\rho)$ with recolouring map~$\rho$ and empty internal tournament (and thus empty boundary tournament and trivial colouring).
  Given a linear NCL-decomposition of~$T$ with width~$k$, then it is routine to prove that~$T$ is
  the internal tournament of the bag obtained as products of atomic bags of the form $B(X,c)$ and $B(\rho)$ obtained from this NLC decomposition.

  Conversely, let~$T$ be the internal tournament of the product $B_1\cdot B_2 \cdot \ldots \cdot B_n$
  where each~$B_i$ is an atomic bag with interval vertex~$v_i$.
  Then~$T$ with vertices labelled by~$\lambda_{B_1\cdot \ldots \cdot B_n}$ can be constructed by a linear NLC decomposition of width~$k+1$, by induction on~$n$.
  Assuming the internal tournament of $B_1 \cdot \ldots \cdot B_{n-1}$ has been constructed in this way,
  we add~$v_{n}$ with the unused colour~$k+1$ and connect it to the existing vertices of colour~$i$ according to the edge between~$i$ and~$v$ in~$B_n$.
  This is a vertex-addition operation in the sense of NLC decomposition.
  Finally, we use a recolouring operation, applying the recolouring function $\rho_{B_n}$ to colours in~$[k]$, and mapping~$k+1$ to the colour of~$v$ in~$B_n$.
\end{proof}

See \cref{fig:decomposition} for some examples. Thanks to \cref{lem:monoid-nlcwidth}, we call by extension any word $B_1\cdot B_2\cdot \ldots \cdot B_n$ of atomic bags a linear decomposition.

\begin{figure}[tp]
  \centering
  \subcaptionbox{%
    Inverted path: all edges are left-to-right except~$i \from i+1$.
  }[.4\textwidth]{
    \begin{tikzpicture}
      \foreach \i in {1,...,5}{
        \node[vertex] (\i) [label=below:$\i$] at (\i,0) {};
      }
      \foreach \i in {1,...,4}{
        \pgfmathtruncatemacro{\j}{\i+1}
        \draw[directed] (\j) -- (\i);
      }
      \foreach \i in {1,...,3}{
        \foreach \j in {3,...,5}{
          \pgfmathtruncatemacro{\k}{\i+1}
          \ifnum \k<\j
            \draw (\i) edge[directed,bend left=30] (\j);
          \fi
        }
      }
    \end{tikzpicture}
  }
  \hfill
  \subcaptionbox{%
    Decomposition of the inverted path.
  }{
    \begin{tikzpicture}[scale=0.6]
      \foreach \i in {0,...,5}{
        \node[fakevertex] (i\i1) at (2*\i,0) {};
        \node[fakevertex] (i\i2) at (2*\i,1) {};
        \draw (2*\i-0.22,-0.33) rectangle (2*\i+0.22,1.33);
      }
      \foreach \i in {1,...,5}{
        \node[vertex] (\i) [label=below:$\i$] at (2*\i-1,-0.5) {};
        \pgfmathtruncatemacro{\j}{\i-1}
        \draw[colormap] (\i) -- (i\i1);
        \draw[colormap] (i\j1) -- (i\i2);
        \draw[colormap] (i\j2) -- (i\i2);
        \draw[directed] (\i) -- (i\j1);
        \draw[directed] (i\j2) -- (\i);
      }
    \end{tikzpicture}
  }

  \subcaptionbox{%
    Rotating tournament: vertices are on a circle with arcs oriented clockwise.
  }[.42\textwidth]{
    \begin{tikzpicture}
      \foreach \i in {1,...,7}{
        \pgfmathtruncatemacro{\a}{360/7*(1-\i)+90}
        \node[vertex] (\i) [label=\a:$\i$] at (\a:1) {};
      }
      \foreach \i/\j in {1/2,2/3,3/4,4/5,5/6,6/7,7/1,1/3,2/4,3/5,4/6,5/7,6/1,7/2,1/4,2/5,3/6,4/7,5/1,6/2,7/3}{
        \draw[directed] (\i) -- (\j);
      }
    \end{tikzpicture}
  }
  \hfill
  \subcaptionbox{%
    Rotating tournament drawn left-to-right following the order of a linear decomposition of width~2.
  }[.5\textwidth]{
    \begin{tikzpicture}[xscale=1.5]
      \foreach \i in {1,...,4}{
        \node[vertex] (\i) [label=below:$\i$] at (\i,-0.5) {};
      }
      \foreach \i in {5,...,7}{
        \node[vertex] (\i) [label=above:$\i$] at (\i-3.5,0.5) {};
      }

      \foreach \i/\j in {1/2,2/3,3/4,4/5,5/6,6/7,7/1,3/5,4/6,6/1,7/2,2/5,3/6,4/7,5/1,6/2,7/3}{
        \draw[directed] (\i) -- (\j);
      }
      \draw (1) edge[directed, bend right=10] (3);
      \draw (1) edge[directed, bend right=10] (4);
      \draw (2) edge[directed, bend right=10] (4);
      \draw (5) edge[directed, bend left=10] (7);
    \end{tikzpicture}
  }

  \subcaptionbox{%
    Decomposition of the rotating tournament.
  }{
    \begin{tikzpicture}[scale=0.7]
      \foreach \i in {0,...,7}{
        \node[fakevertex] (i\i1) at (2*\i,0) {};
        \node[fakevertex] (i\i2) at (2*\i,1) {};
        \draw (2*\i-0.2,-0.3) rectangle (2*\i+0.2,1.3);
      }
      \foreach \i in {1,...,4}{
        \node[vertex] (\i) [label=below:$\i$] at (4*\i-3,-0.5) {};
        \pgfmathtruncatemacro{\n}{2*\i-1}
        \pgfmathtruncatemacro{\p}{2*\i-2}
        \draw[colormap] (\i) -- (i\n1);
        \draw[directed] (i\p1) -- (\i);
        \draw[directed] (\i) -- (i\p2);
      }
      \foreach \i in {5,...,7}{
        \node[vertex] (\i) [label=above:$\i$] at (4*\i-17,1.5) {};
        \pgfmathtruncatemacro{\n}{2*\i-8}
        \pgfmathtruncatemacro{\p}{2*\i-9}
        \draw[colormap] (\i) -- (i\n2);
        \draw[directed] (i\p2) -- (\i);
        \draw[directed] (\i) -- (i\p1);
      }
      \foreach \i in {0,...,6}{
        \pgfmathtruncatemacro{\j}{\i+1}
        \draw[colormap] (i\i1) -- (i\j1);
        \draw[colormap] (i\i2) -- (i\j2);
      }
    \end{tikzpicture}
  }
  \caption{Two examples of linear decompositions.}
  \label{fig:decomposition}
\end{figure}
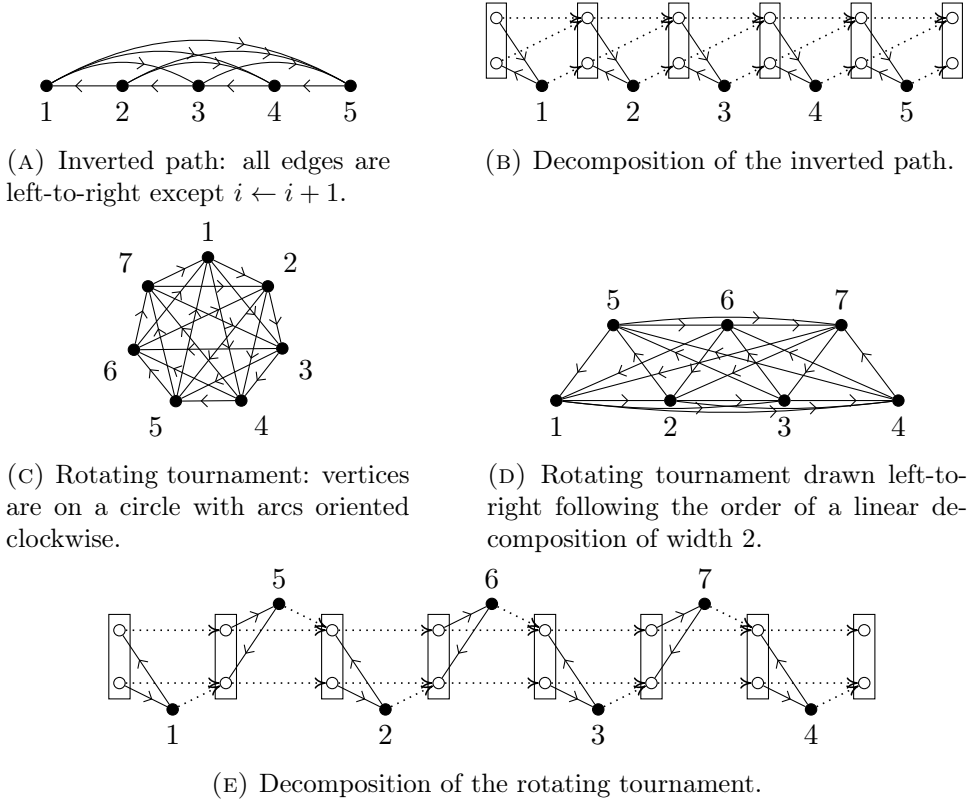

\subsection{Types and abstraction}\label{types and finite presentation}
We next define some finite \emph{abstraction} of bags.
The first step is to define \emph{types} of vertices inside a bag.

Consider a bag~$B$ and an internal vertex~$v \in V(B)$.
The type of~$v$ describes its interactions with bags which may be multiplied to the left or right of~$B$.
Specifically, $\type_B(v)$ consists of the following information:
\begin{enumerate}
 \item the colour~$\lambda_B(v)$, and
 \item the direction of edges $i \to v$ or $i \from v$ for each input vertex~$i \in [k]$.
\end{enumerate}

We use~$\sigma,\tau,\dots$ to denote vertex types.
When the order~$k$ of the bag is fixed, there are only~$k \cdot 2^k$ distinct vertex types.

Remark that the type of a given vertex~$v$ may change when multiplying bags:
given bags $A,B,C$ and~$v \in V(B)$, in general $\type_{B}(v)$ may differ from~$\type_{ABC}(v)$.
However, $\type_{ABC}(v)$ only depends on~$\type_B(v)$ and the recolouring functions~$\rho_A,\rho_C$:
\begin{lemma}
  \label{lem:type-equiv-congruence} Let $k$ be a positive integer. For every two recolouring functions $g:[k]\to [k]$ and $h:[k]\to [k]$, there is a function
  $f_{g,h}$ from vertex types to vertex types such that, for any three bags $A,B,C$ of order $k$  and any $v\in V(B)$,
  \[ \type_{ABC}(v) = f_{\rho_A,\rho_C}(\type_B(v)). \]
\end{lemma}
\begin{proof}
  Applying the definition of product of bags, we have the following:
  The colour of~$v$ in the product~$ABC$ is $\lambda_{ABC}(v) = \rho_C(\lambda_B(v))$.
  For an input vertex~$i \in [K]$, there is an edge $i \to v$ in $\edgebnd(ABC)$ if and only if $\rho_A(i) \to v$ is an edge in~$\edgebnd(B)$.
  The above defines $\type_{ABC}(v)$, and only depends on~$\type_B(v)$, $\rho_A$, and~$\rho_C$.
\end{proof}
In particular, this means that if vertices~$u,v \in V(B)$ have the same type $\type_B(u) = \type_B(v)$ in~$B$, then they also have the same type $\type_{ABC}(u) = \type_{ABC}(v)$ in~$ABC$.

Next, we define the abstraction of a bag.
If~$B$ is a bag and~$\tau$ is a vertex type, we denote by $\tau(B) \eqdef \setst{v \in V(B)}{\type_B(v) = \tau}$ the set of vertices with type~$\tau$ in~$B$.
The type~$\tau$ is \emph{inhabited} in~$B$ if~$\tau(B)$ is non-empty.
The abstraction of~$B$, denoted by~$\abstr{B}$, now consists of the following information:
\begin{enumerate}
 \item the recolouring function~$\rho_B : [k] \to [k]$,
 \item the set of inhabited vertex types in~$B$, and
 \item for each pair of inhabited types~$\sigma,\tau$,
   whether~$\sigma(B)$ and~$\tau(B)$ are homogeneous or not in the internal tournament of~$B$,
   and when they are, the direction of edges $\sigma(B) \to \tau(B)$ or $\sigma(B) \from \tau(B)$.
\end{enumerate}

We use~$\alpha,\beta,\dots$ to denote abstractions of bags.
For fixed order~$k$, the number of bag abstractions is some constant~$2^{2^{O(k)}}$.

A routine check using crucially \cref{lem:type-equiv-congruence} shows that for any bags~$B,B'$ of order~$k$,
$\abstr{B \cdot B'}$ only depends on~$\abstr{B}$ and~$\abstr{B'}$.
This implicitly defines a monoid structure on the set of bag abstractions,
which is a quotient of the monoid of bags.
Equivalently, the map $B \mapsto \abstr{B}$ from bags to abstractions is a monoid homomorphism. We denote by~$M_k$ this monoid of abstractions of bags of order~$k$.
We summarise all these properties in the following.

\begin{lemma}\label{lem:bag-types-hom}
  For any~$k$, the set~$M_k$ of abstractions has size bounded by~$2^{2^{O(k)}}$,
  and there is an associative operation $\tilde{\cdot}$ such that
  \[ \abstr{B_1\cdot B_2}=\abstr{B_1} \tilde{\cdot} \abstr{B_2}. \]
\end{lemma}

\begin{proof}
  The number of recolouring functions is~$k!$.
  There are $k\cdot 2^k$ possible vertex types, hence $2^{k\cdot 2^k}$ possibilities for which vertex types are inhabited,
  and $3^{(k\cdot 2^k)^2}$ choices for the information on homogeneity and edge directions between vertex types.
  Thus, the size of $M_k$ is bounded by $2^{2^{O(k)}}$.

  Let us now prove the existence of the product~$\tilde{\cdot}$ for~$M_k$.
  That is, we need to show that $\abstr{B_1 \cdot B_2}$ depends only on~$\abstr{B_1}$ and~$\abstr{B_2}$.
  First, the recolouring function of $B_1\cdot B_2$ is by definition the composition $\rho_{B_2}\circ
  \rho_{B_1}$.
  Next, by \cref{lem:type-equiv-congruence}, there is a map~$f_1$ on vertex types, depending only on~$\rho_{B_2}$, such that $\type_{B_1B_2}(v) = f_1(\type_{B_1}(v))$ for all internal vertices~$v$ of~$B_1$;
  and a similar map~$f_2$ for~$B_2$ depending only on~$\rho_{B_1}$.
  Then a vertex type~$\sigma$ is inhabited in $B_1 \cdot B_2$ if and only if
  some type in~$f_1^{-1}(\sigma)$ or~$f_2^{-1}(\sigma)$ is inhabited in~$B_1$ or~$B_2$ respectively.

  Finally, consider two vertex types~$\sigma,\tau$, for which we want to test homogeneity in~$B_1 \cdot B_2$.
  For $\sigma_1,\tau_1$ inhabited in~$B_1$ with $f_1(\sigma_1) = \sigma$, $f_1(\tau_1) = \tau$,
  we know from~$\abstr{B_1}$ the direction of edges between~$\sigma_1(B_1)$ and~$\tau_1(B_1)$.
  The same holds for types in~$B_2$ mapping to~$\sigma,\tau$ through~$f_2$.
  Next, if~$\sigma_1,\tau_2$ are inhabited in~$B_1$ and~$B_2$ respectively, and $f_1(\sigma_1) = \sigma$, $f_2(\tau_2) = \tau$,
  then the direction of edges between $\sigma_1(B_1)$ and $\tau_2(B_2)$ (which are always homogeneous) is entirely determined by~$\sigma_1$ and~$\tau_2$ themselves.
  The same holds when swapping the roles of~$\sigma$ and~$\tau$.
  To test whether~$\sigma$ and~$\tau$ are homogeneous, it suffices to consider all pairs of inhabited vertex types from combinations of~$B_1,B_2$ as above, and check that the edge direction is either always from~$\sigma$ to~$\tau$, or vice versa.

  We have then shown that one can determine the abstraction of $B_1\cdot B_2$
  from the abstractions of $B_1$ and $B_2$, \ie there is an operation $\tilde{\cdot}$ such that $\abstr{B_1\cdot B_2}=\abstr{B_1}\tilde{\cdot} \abstr{B_2}$.
  Because $\cdot$ is associative, $\tilde{\cdot}$ must also be associative.
\end{proof}

\subsection{Simon's Factorisation Forest Theorem}
\label{subsection: simon}
Given a linear decomposition $B = B_1B_2 \dots B_n$,
Simon's theorem gives a \emph{factorisation forest},
which is a more structured way to construct~$B$ from~$B_1,\dots,B_n$,
subject to some restrictions relative to a finite semigroup~$S$ of our choice---in our case the monoid of abstractions.

The general setting of Simon's theorem is the following.
Consider a (usually infinite) semigroup~$\Sigma$ and a generating set~$A \subset \Sigma$,
a finite semigroup~$S$, and a homomorphism $\phi : \Sigma \to S$.
In our case, $\Sigma$ are the bags, $A$ the atomic bags, $S$ the abstractions, and $\phi(B) = \abstr{B}$.
Simon's theorem deals with two kinds of factorisations of elements $w \in \Sigma$:
\begin{description}
  \item[binary factorisation] $w = w_1\cdot w_2$ for $w_1,w_2 \in \Sigma$ with no further restriction.
  \item[idempotent factorisation] $w = w_1 \cdot \ldots \cdot w_n$, where the number~$n$ of factors is unrestricted,
    but the factors must all map to the same $e = \phi(w_i)$ in~$S$,
    which furthermore must be idempotent, meaning $e \cdot e = e$.
\end{description}
Simon's theorem states that any $w \in \Sigma$ can be reduced down to generators in~$A$
by applying the previous two operations nested only up to depth bounded by a linear function of~$\card{S}$.
This process is described as a \emph{factorisation forest} for~$w$, see \cref{fig:factorisation} for an illustration.

Let us introduce some terminology allowing to count binary and idempotent factorisations separately.
A factorisation forest has depth~$(p,q)$ if the nesting depth of idempotent operations is at most~$p$,
and that of binary operations is at most~$q$. Formally:
\begin{itemize}
  \item Generators $a \in A$ have a factorisation forest of depth~$(0,0)$.
  \item If~$w_1,w_2$ have factorisation forests of depth~$(p_1,q_1)$ and~$(p_2,q_2)$ respectively,
    then~$w_1 \cdot w_2$ has a factorisation forest of depth
    \[ (\max\{p_1,p_2\}, \max\{q_1,q_2\} + 1). \]
  \item If~$w_1,\dots,w_n \in \Sigma$ satisfy $\phi(w_i) = e$ for some idempotent~$e \in S$ and all~$1\leq i\leq n$,
    and there are~$p,q \in \Nn$ such that each~$w_i$ has a factorisation forest of depth~$(p_i,q_i)$,
  then~$w_1 \cdot \ldots \cdot w_n$ has a factorisation forest of depth~$(\max \{p_1,\dots,p_n\}+1,\max \{q_1,\dots,q_n\})$.
\end{itemize}
Since this definition depends on the choice of semigroup~$S$ and morphism~$\phi$,
we will call it a \emph{factorisation forest over~$\phi$} to disambiguate.

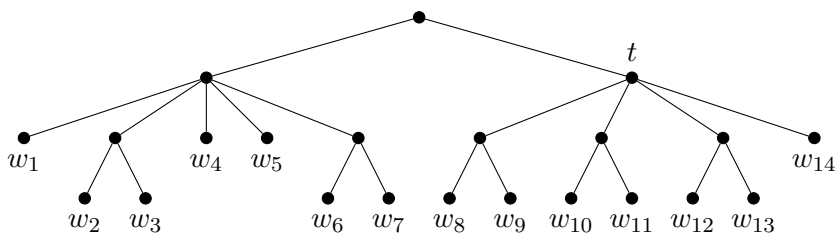
\begin{figure}[ht]
  \centering
  \begin{tikzpicture}[scale=0.8]
    \node[vertex] (r) at (-0.5,0) {};
    \node[vertex] (x) at (-4,-1) {};
    \node[vertex] (x1) [label=below:$w_1$] at (-7,-2) {};
    \node[vertex] (x2) at (-5.5,-2) {};
    \node[vertex] (x21) [label=below:$w_2$] at (-6,-3) {};
    \node[vertex] (x22) [label=below:$w_3$] at (-5,-3) {};
    \node[vertex] (x3) [label=below:$w_4$] at (-4,-2) {};
    \node[vertex] (x4) [label=below:$w_5$] at (-3,-2) {};
    \node[vertex] (x5) at (-1.5,-2) {};
    \node[vertex] (x51) [label=below:$w_6$] at (-2,-3) {};
    \node[vertex] (x52) [label=below:$w_7$] at (-1,-3) {};

    \node[vertex] (y) [label=above:$t$] at (3,-1) {};
    \node[vertex] (y1) at (0.5,-2) {};
    \node[vertex] (y11) [label=below:$w_8$] at (0,-3) {};
    \node[vertex] (y12) [label=below:$w_9$] at (1,-3) {};
    \node[vertex] (y2) at (2.5,-2) {};
    \node[vertex] (y21) [label=below:$w_{10}$] at (2,-3) {};
    \node[vertex] (y22) [label=below:$w_{11}$] at (3,-3) {};
    \node[vertex] (y3) at (4.5,-2) {};
    \node[vertex] (y31) [label=below:$w_{12}$] at (4,-3) {};
    \node[vertex] (y32) [label=below:$w_{13}$] at (5,-3) {};
    \node[vertex] (y4) [label=below:$w_{14}$] at (6,-2) {};

    \foreach \x/\y in {r/x,x/x1,x/x2,x2/x21,x2/x22,x/x3,x/x4,x/x5,x5/x51,x5/x52,r/y,y/y1,y1/y11,y1/y12,y/y2,y2/y21,y2/y22,y/y3,y3/y31,y3/y32,y/y4}{
      \draw (\x) -- (\y);
    }
  \end{tikzpicture}
  \caption{%
    Simon's factorisation of a word $w_1 \dots w_{14}$ represented at a tree.
    At the leaves are letters~$w_i$ in the chosen generating set~$A$.
    Each internal node corresponds to a binary or idempotent operation.
    E.g.\ since~$t$ is not binary, it must be idempotent:
    $\phi(w_8w_9) = \phi(w_{10}w_{11}) = \phi(w_{12}w_{13}) = \phi(w_{14}) = e$,
    and further this~$e$ must satisfy $e \cdot e = e$.
  }
  \label{fig:factorisation}
\end{figure}

Simon's theorem can now be precisely stated as follows.
\begin{theorem}[Simon's Factorisation Forest Theorem, \cite{simon1990factorization,kufleitner2008height}]
  \label{thm:simon}
  Consider a semigroup~$\Sigma$ generated by~$A \subset \Sigma$, and a homomorphism $\phi : \Sigma \to S$ to a finite semigroup.
  Then all~$w \in \Sigma$ have factorisation forests over~$\phi$ of depth at most~$(\card{S},2\card{S})$.
\end{theorem}
The original proof by Simon~\cite{simon1990factorization} shows that there are factorisation forests of total depth~$9\card{S}$.
Kufleitner~\cite{kufleitner2008height} improved this bound to~$3\card{S}$,
and the proof in fact uses only~$\card{S}$ idempotent factorisations and~$2\card{S}$ binary ones.

In the case of bags, we will apply \cref{thm:simon} to the monoid~$\mathfrak{C}_K$ with atomic bags as generators,
and the homomorphism $\abstr{\cdot} : \mathfrak{C}_K \to M_K$.
Since we will only consider factorisation forests over this homomorphism,
we will omit the `over~$\abstr{\cdot}$' qualifier in the rest of this work.

\subsection{Abstractions and vertex deletion}
The proof of our main result proceeds by induction over the depth of a given factorisation forest.
At some point during this induction, we may need to restrict the bag considered to a subset of vertices.
In general, this does not preserve the factorisation depth. Indeed, deleting vertices in a bag may change its abstraction,
and what was an idempotent factorisation may thus stop being one.
This section describes a sufficient condition to control the depth of factorisation forests when deleting vertices.

Given a bag~$B$ and a subset~$X \subseteq V(G)$ of internal vertices,
we denote by~$B[X]$ the bag obtained by deleting any internal vertex of~$B$ outside~$X$, and restricting the edges and colouring map of~$B$ to~$X$.
If~$\Gamma$ is a set of vertex types, then we denote by
$\Gamma(B) = \bigcup_{\tau \in \Gamma} \tau(B)$ the set of vertices with type in~$\Gamma$,
and we say that~$\Gamma(B)$ is a \emph{union of types}.
When deleting vertices from bags, we will only delete unions of types,
which is sufficient to control the factorisation depth.
\begin{lemma}\label{lem:type-deletion} Let $k$ be a positive integer and let $B$ be a bag in $\mathfrak{C}_k$.
  If~$B$ has a factorisation forest of depth~$(p,q)$ and~$X \subset V(B)$ is some union of types,
  then~$B[X]$ has a factorisation forest of depth~$(p,q+2p)$.
\end{lemma}
Before proving \cref{lem:type-deletion}, let us make two very simple observations.
The first follows directly from the definition of abstractions.
\begin{remark}\label{rmk:type-union-preserves-type}
  If~$\Gamma$ is a set of vertex types, and~$B_1,B_2$ are two bags with $\abstr{B_1} = \abstr{B_2}$, then
  \[ \abstr{B_1[\Gamma(B_1)]} = \abstr{B_2[\Gamma(B_2)]}. \]
\end{remark}

Secondly, recall that the type of a vertex~$x$ changes depending on the bag containing~$x$ considered:
thus in general~$\Gamma(B) \neq \Gamma(ABC) \cap V(B)$.
Nonetheless, \cref{lem:type-equiv-congruence} shows that for bags~$A,B,C$ and~$x \in V(B)$,
$\type_{ABC}(x)$ only depends on~$\type_B(x)$, $\abstr{A}$, and~$\abstr{C}$.
Thus,
\begin{remark}\label{rmk:type-union-congruence}
  For any bags~$A,B,C$ and set of vertex types~$\Gamma$, there exist~$\Gamma'$ such that
  $\Gamma(ABC) \cap V(B) = \Gamma'(B)$.
  Furthermore, $\Gamma'$ only depends on~$\Gamma$, and $\abstr{A},\abstr{C}$.
\end{remark}

\begin{proof}[Proof of \cref{lem:type-deletion}]
  We prove the result by induction on the factorisation forest of~$B$ with depth~$(p,q)$.
  When~$B$ is atomic, the result is trivial.
  In the binary case, assume that~$B = B_1\cdot B_2$ where~$B_1$ and $B_2$ have each a factorisation forest of depth at most~$(p,q-1)$.
  By \cref{rmk:type-union-congruence}, $X_i \eqdef X \cap V(B_i)$ is also a union of types in~$B_i$ for $i \in \{1,2\}$.
  Hence, by induction we find factorisation forests of depth~$(p,q+2p-1)$ for~$B_1[X_1]$ and $B_2[X_2]$,
  which combine into one of depth~$(p,q+2p)$ for~$B[X]$.

  Consider now the idempotent case~$B = B_1\cdot \ldots \cdot B_n$,
  where all bags~$B_i$ for~$i \in [n]$ have idempotent abstraction~$\alpha$,
  and have factorisation forests of depth at most~$(p-1,q)$.
  Call~$\Gamma$ the set of vertex types defining~$X$, \ie $X = \Gamma(B)$,
  and for each $i \in [n]$, define $X_i \eqdef X \cap V(B_i)$.

  By \cref{rmk:type-union-congruence}, there is a set of vertex types~$\Gamma'$ depending only on~$\Gamma$ and~$\alpha$,
  such that for all bags~$A_1,A_2,A_3$ of abstraction~$\alpha$ we have
  \begin{equation}\label{eq:gamma-prime}
    \Gamma(A_1A_2A_3) \cap V(A_2) = \Gamma'(A_2).
  \end{equation}
  Then, by \cref{rmk:type-union-preserves-type}, there is an abstraction~$\alpha'$, depending again only on~$\Gamma$ and~$\alpha$,
  such that the abstraction of $A_2[\Gamma'(A_2)]$ is~$\alpha'$.
  \begin{claim}\label{claim:type}
    For~$2 \le i \le n-1$, the abstraction of the restriction $B_i[X_i]$ is~$\alpha'$.
  \end{claim}
  \begin{claimproof}
    Immediate by choice of~$\alpha'$ when applying the previous remarks with $A_1 = B_1 \dots B_{i-1}$, $A_2 = B_i$, and $A_3 = B_{i+1} \dots B_n$.
  \end{claimproof}

  \begin{claim}\label{claim:idempotent}
    The abstraction~$\alpha'$ is idempotent.
  \end{claim}
  \begin{claimproof}
    Consider this time four arbitrary bags~$A_1,\dots,A_4$, each with abstraction~$\alpha$.
    Denote~$Y = \Gamma(A_1 A_2 A_3 A_4)$, and for each $i\in \{1,2,3,4\}$, let $Y_i \eqdef Y \cap V(A_i)$.

    Grouping them as~$A_1 A_2 (A_3 A_4)$ and applying~\eqref{eq:gamma-prime}, we have
    \[ Y_2 = \Gamma(A_1 A_2 A_3 A_4) \cap V(A_2) = \Gamma'(A_2). \]
    By instead grouping the four bags as $(A_1A_2) A_3 A_4$, or~$A_1 (A_2A_3) A_4$,
    we similarly find that $Y_3 = \Gamma'(A_3)$ and $Y_2 \cup Y_3 = \Gamma'(A_2A_3)$.
    Then, by choice of~$\alpha'$, the bags $A_2[Y_2]$, $A_3[Y_3]$, and $(A_2A_3)[Y_2 \cup Y_3]$ all have abstraction~$\alpha'$.
    Therefore,
    \begin{align*}
    \alpha' \tilde{\cdot} \alpha' & = \abstr{A_2[Y_2]} \, \tilde{\cdot} \, \abstr{A_3[Y_3]} \\
                      & = \abstr{(A_2A_3)[Y_2 \cup Y_3]} = \alpha'.
    \end{align*}
    Thus $\alpha'$ is idempotent.
  \end{claimproof}

  To conclude, we apply the induction hypothesis to each~$B_i[X_i]$ for $i \in [n]$, yielding factorisation forests of depth~$(q-1,p+2q-2)$.
  By \cref{claim:type,claim:idempotent}, each term in the product $B_2[X_2] \cdot \ldots \cdot B_{n-1}[X_{n-1}]$ has abstraction~$\alpha'$, which is idempotent.
  This is thus an idempotent factorisation of depth~$(q,p+2q-2)$.
  With two additional binary factorisations to add~$B_1[X_1]$ and~$B_n[X_n]$,
  we obtain a factorisation forest of depth~$(q,p+2q)$ for $B_1[X_1]\cdot  \ldots \cdot B_n[X_n] = (B_1 \dots B_n)[X]$.
\end{proof}

\section{Transducing an ordering}
\label{sec:main}
In this section, we show that orderings of bounded cut-rank can be \fo transduced.
Throughout the section, we work with (non-deterministic) extension transductions~$\Phi$ with signature $\{E\} \to \{E,<\}$, where~$E,<$ are binary relational symbols.
Explicitly, such a transduction takes as input a (directed) graph~$G = (V,E)$, and non-deterministically outputs structures $(V,E,<)$,
i.e.\ the same graph~$(V,E)$ plus some binary relation~$<$ expected to be a linear ordering.
\begin{theorem}
  \label{thm:transduce-ordering}
  For any~$k \in \Nn$, one can compute an extension \fo transduction $\Phi_k$ with signature $\{E\} \to \{E,<\}$,
  and~$k' \in \Nn$ such that
  \begin{enumerate}
    \item for any graph~$G=(V,E)$, every structure in~$\Phi_k(G)$ is~$(V,E,<)$ for some linear ordering~$<$ of~$G$ with cut-rank at most~$k'$, and
    \item if~$T$ is a tournament with linear NLC-width at most~$k$, then $\Phi_k(T)$ contains at least one such structure.
  \end{enumerate}
\end{theorem}

The proof will follow the same line as in \cite{bojanczyk2018cliquewidth,bojanczyk2016definability} and will be by induction on the height of the Simon's
Factorisation Forest of the words generating bags in $\mathfrak{C}_k$.
We first show in \cref{subsec:check} that the first condition of \cref{thm:transduce-ordering} is easy to ensure.
The main induction is explained in \cref{subsec:induction} and then we proceed to the proof of the main technical lemma in \cref{subsec:main-lemma}.

\subsection{Checking the cut-rank of an ordering}\label{subsec:check}
We first handle the first condition of \cref{thm:transduce-ordering}:
ensuring that the output of~$\Phi_k$ only contains orderings of cut-rank at most~$k'$.

Consider a bi-partition $V(G) = X \uplus Y$ of the vertices of a graph~$G$.
For~$x \in X$, denote by $N_Y(x) = \{y \in Y : (x,y) \in E(G)\}$ its neighbourhood in~$Y$.
By definition, $\rk(X; Y) \le k$ if and only if there are at most~$k$ vertices $x_1,\dots,x_k \in X$ such that
for any~$x' \in X$, there is some $B \subseteq [k]$ satisfying $N_Y(x') = \bigoplus_{b \in B} N_Y(b)$,
where~$\oplus$ denotes sum modulo~2 (or symmetric difference) of sets.
If~$X$ is given as a unary predicate, it is routine using this characterisation to write an \fo formula with~$k+1$ quantifiers asserting that $\rk(X;Y) \le k$.

Consider now an ordering $x_1 < \dots < x_n$ of~$V(G)$.
Replacing the condition $x \in X$ by $x \le x_i$ in the previous argument
yields a formula~$\phi(x_i)$ checking that $\rk(\{x_1,\dots,x_i\}; \{x_{i+1},\dots,x_n\}) \le k$.
Then, the formula $\forall x_i.\ \phi(x_i)$ checks that the ordering~$<$ has cut-rank at most~$k$.
Given any extension transduction~$\Phi$ with signature $\{E\} \to \{E,<\}$,
one can add to~$\Phi$ a filtering step which checks that (1) the relation~$<$ is indeed a linear ordering,
and (2) this ordering has cut-rank at most~$k'$, using the previously described formula.
This ensures that~$\Phi$ satisfies the first condition of \cref{thm:transduce-ordering},
and if~$\Phi(T)$ did contain $(V,E,<)$ for some ordering~$<$ of cut-rank at most~$k'$, then it will still be there after filtering.

Henceforth, we only focus on the second condition of \cref{thm:transduce-ordering}:
building an extension transduction~$\Phi_k$ such that on any~$T$ of linear NLC-width~$k$, the output~$\Phi_k(T)$ contains at least one ordering of cut-rank at most~$k'$.

\subsection{Main induction}\label{subsec:induction}
By \cref{lem:monoid-nlcwidth}, any tournament~$T$ with linear NLC-width~$k$ is also the internal tournament of some bag~$B$ in~$\mathfrak{C}_k$.
Furthermore, by \cref{thm:simon}, this bag~$B$ has factorisation depth at most~$(|M_k|,2|M_k|)$,
where~$M_k$ is the monoid of abstractions of arity~$k$, whose size is a function of~$k$ only.
Thus, \cref{thm:transduce-ordering} follows from the next statement and the arguments of \cref{subsec:check}.
\begin{lemma}
  \label{lem:transduce-ordering}
  For any~$k,p,q \in \Nn$, one can compute $f(k,p,q) \in \Nn$ and
  an extension \fo transduction $\Phi_{(p,q)}^k : \{E\} \to \{E,<\}$ such that
  if~$T = (V,E)$ is the internal tournament of a bag $B \in \mathfrak{C}_k$ of factorisation depth at most~$(p,q)$,
  then there is some ordering~$<$ of~$T$ with cut-rank at most~$f(p,q,k)$ such that $(V,E,<) \in \Phi_{(p,q)}^k(T)$.
\end{lemma}
We once again insist that the output $\Phi_{(p,q)}^k(T)$ needs to contain some linear ordering of~$T$ with the required cut-rank,
but may also contain other structures where~$<$ is interpreted arbitrarily (possibly not even as a linear ordering).
Thus we will show that if non-deterministic steps follow the `right' choices,
then the transduction produces the desired ordering, while ignoring anything resulting from a `wrong' choice.
We phrase this as the non-deterministic step \emph{guessing} the desired colouring.

Let us begin the proof of \cref{lem:transduce-ordering}, by induction on the pair~$(p,q)$ ordered lexicographically.
Choose the bound~$f(k,p,q)$ to satisfy:
\begin{align*}
  f(k,0,0) &= k, \\
  f(k,p,q) &\ge f(k,p,q-1)+k && \text{if $q>0$, and}\\
  f(k,p,q) &\ge f(k,p-1,2p+q)+2k\cdot (2^k+1) && \text{if $p>0$.}
\end{align*}
Consider a bag~$B \in \mathfrak{C}_K$ whose internal tournament is~$T$, and with factorisation depth~$(p,q)$.

In the base case~$p = q = 0$, the bag~$B$ is atomic, and the result is trivial.
Assume now that~$p>0$ or~$q>0$, so that the Simon's factorisation of~$B$
has either a binary or an idempotent operation at the root.
The transduction~$\Phi_{(p,q)}^k$ starts by guessing which of these two cases occurs.
This is simply a boolean guess, which can be simulated with the non-deterministic colouring operation.
Depending on this `binary' or `idempotent' guess, the transduction continues with the corresponding branch described in the next two paragraphs;
this means that the formulas used in the transduction~$\Phi_{(p,q)}^k$ are of the form
$(idempotent \land \phi_I) \lor (\lnot idempotent \land \phi_B)$,
where $idempotent$ represents the previous boolean guess,
and $\phi_B,\phi_I$ describe the appropriate transduction for the binary and idempotent cases respectively.

\subsubsection*{Binary case}
\label{binary case of the main induction}
Assume $B = B_1 \cdot B_2$ where~$B_1,B_2$ have factorisations of depth at most~$(p,q-1)$.
We use a non-deterministic colouring to guess the bipartition~$V(B_1),V(B_2)$.
Then, by restricting the quantifiers to~$V(B_i)$, one can simulate applying the transduction~$\Phi_{(p,q-1)}^k$ obtained by induction to each~$T[V(B_i)]$.
Thus, we obtain linear orderings~$<_1,<_2$ of cut-rank at most~$f(k,p,q-1)$ for~$T[V(B_1)]$ and~$T[V(B_2)]$ respectively.

Define~$<$ on~$T$ by~$V(B_1) < V(B_2)$, where~$<$ coincides with~$<_i$ inside~$V(B_i)$.
One can transduce~$<$ from~$<_1,<_2$ and the bipartition~$V(B_1),V(B_2)$.
Since we are dealing with bags of arity~$k$, we have $\rk(V(B_1); V(B_2)) \le k$.
Since~$<_i$ has cut-rank at most~$f(k,p,q-1)$ in~$T[V(B_i)]$ for~$i \in \{1,2\}$,
this implies that~$<$ has cut-rank at most~$f(k,p,q-1)+k\leq f(k,p,q)$.

\subsubsection*{Idempotent case}
Assume now we are given an idempotent factorisation,
that is $B = B_1\cdot \ldots \cdot B_n$ where~$\abstr{B_i} = \alpha$ for all~$1\leq i\leq n$, with~$\alpha$ idempotent,
and each~$B_i$ has factorisation depth at most~$(p-1,q)$.

Consider the quasi-ordering~$\qle$ of~$V(B)$ defined by~$V(B_1) \qle \dots \qle V(B_n)$,
and call~$\sim$ the equivalence relation whose classes are $V(B_1),\dots, V(B_n)$.
Suppose in a first time that we are given~$\qle$ (transducing it will be the core of this proof).
Using parallel application (\cref{lem:parallel-transduction}), this allows to apply~$\Phi_{(p-1,q)}^k$ to each~$B_i$ simultaneously,
yielding a relation~$<$ which inside each~$B_i$ is interpreted as a linear ordering of cut-rank at most~$f(k,p-1,q)$.

Define~$<'$ to be the linear ordering which coincides with~$<$ inside each bag~$B_i$, and with~$\qle$ between the bags.
Clearly~$<'$ can be transduced from~$<$ and~$\qle$, and \cref{lem:lexico-rankwidth} gives a bound on the cut-rank of~$<'$.

We are thus only left with the core problem of transducing the quasi-ordering~$\qle$.
Unfortunately, it is not always possible to transduce exactly~$\qle$.
For instance~$B$ might be partitioned into two homogeneous sets~$X_1,X_2$, each intersecting all~$B_j$.
The ordering~$\qle$ then alternates an unbounded number of times between~$X_1$ and~$X_2$.
This means that across the bipartition $(X_1,X_2)$, tournament edges are a relation with bounded rank, but the ordering~$\qle$ has unbounded rank.
It follows by compositionality arguments that~$\qle$ cannot be transduced from the tournament.

In this case, we need to reorganise the decomposition of~$B$
by first separating~$X_1$ and~$X_2$, and then decomposing each of them.
The precise statement we prove is thus the following.

\begin{lemma}
  \label{lem:idempotent-partition}
  Given~$k$, one can compute an extension \fo transduction $\Psi_k$ with signature $\{E\} \to \{E,\qle\}$ satisfying the following.
  If~$T = (V,E)$ is the internal tournament of a bag $B\in \mathfrak{C}_k$ with factorisation depth at most~$(p,q)$,
  then there is a quasi-ordering~$\qle$ of~$V$ with equivalence classes $X_1 \qle \dots \qle X_m$
  such that
  \begin{enumerate}
    \item \label{item:fact-depth} for each~$i \in [m]$, $T[X_i]$ is the internal tournament of a bag in $\mathfrak{C}_k$ having factorisation depth at most~$(p-1,q+2p)$,
    \item \label{item:part-rank} for each~$i \in [m]$,
      \[ \rk(X_1 \cup \dots \cup X_i;\ X_{i+1} \cup \dots \cup X_m) \le k\cdot (2^k+1), \]
    \item \label{item:transduction} and $(V,E,\qle) \in \Psi_{(p,q)}^k(T)$.
  \end{enumerate}
\end{lemma}

Assuming \cref{lem:idempotent-partition}, we conclude as follows.
From the quasi-ordering~$\qle$, it is simple to transduce the equivalence relation~$\sim$ with equivalence classes $X_1,\dots,X_m$.
Recall that we do induction on~$(p,q)$ ordered lexicographically,
hence we can assume the transduction~$\Phi^k_{(p-1,q+2p)}$ is already defined.
Using parallel application (\cref{lem:parallel-transduction}), this transduction can be applied to each~$X_i$ in parallel.
This yields a relation~$<$ which inside each~$T[X_i]$ is a linear ordering with cut-rank at most~$f(k,p-1,q+2p)$.
Define~$<'$ to coincide with~$<$ inside each~$X_i$ and with~$\qle$ between the different~$X_i$,
which is easily transduced from~$<,\qle$.
Using \cref{lem:lexico-rankwidth}, and condition~\ref{item:part-rank} of \cref{lem:idempotent-partition},
we obtain that~$<'$ has cut-rank at most
\[ f(k,p-1,q+2p)+2k\cdot (2^k+1) \le f(k,p,q). \]
This concludes the idempotent case, and thus the proof of \cref{lem:transduce-ordering}.

\subsection{Proof of Lemma \ref{lem:idempotent-partition}}\label{subsec:main-lemma}
For this entire section, we fix some~$k \in \Nn$,
and consider a factorisation $B = B_1 \cdot \ldots \cdot B_n$,
where all~$B_i$ have the same idempotent abstraction~$\alpha$, and have factorisation depth at most~$(p-1,q)$.
We also call $T = (V,E)$ the internal tournament of~$B$.

Recall that \cref{lem:idempotent-partition} asks for a transduction depending only on~$k$.
As the very first step of this transduction, one may use non-deterministic colouring to guess the abstraction~$\alpha$
(for which there are only finitely many possibilities for a fixed~$k$, by \cref{lem:bag-types-hom}).
We may thus instead allow the transductions constructed in this section to depend on both~$k$ and~$\alpha$
(but not~$B$, the factorisation, or~$p,q$).

For a vertex~$x \in V(B_i)$, we call $\idx(x) = i$ the \emph{index} of $x$.
Our goal is to answer (with an \fo transduction) the question: given~$x,y$, is~$\idx(x) \le \idx(y)$?
To be formal, for a subset~$X$ of~$V(B)$, we say that a transduction~$\Phi$ \emph{orders~$X$ according to~$\qle$}
if~$\Phi$ is an extension \fo transduction, and there is some $(V,E,\qle') \in \Phi(T)$ such that~$\qle$ and~$\qle'$ coincide on~$X$.

Using the non-determinism of transductions, one may show that it suffices to consider vertices whose indices differ by at least some constant.
That is, say that~$\Phi$ \emph{approximately} orders~$X$ according to~$\qle$
if once again~$\Phi$ is an extension \fo transduction, and there is $(V,E,\qle') \in \Phi(T)$ such that~$\qle,\qle'$
coincide for any pair $x,y \in X$ that satisfies~$|\idx(x) - \idx(y)| \ge 2$.
When~$X$ hits all bags of the factorisation, if we can approximately order~$X$, then we can also order it:
\begin{lemma}
 \label{lem:ordering-approx}
 Let~$X$ be a subset of vertices containing at least one vertex for each possible index.
 For any \fo transduction~$\Phi$, there is a second transduction~$\Psi$ such that
 if~$\Phi$ approximately orders~$X$ according to~$\qle$, then~$\Psi$ orders~$X$ according to~$\qle$.
\end{lemma}

\begin{proof}
  Using non-deterministic colouring, $\Psi$ guesses the index of every vertex modulo~5.
  Let~$x$ and $y$ be two vertices in~$X$ with indices~$i$ and~$j$ respectively.
  If~$i$ and $j$ differ by at least~2 modulo~5, then we can by assumption test whether~$x \qle y$ using~$\Phi$.
  Assume thus that~$j$ is one of~$\{i-1,i,i+1\}$ modulo~5.
  Then, there exists~$k \in \Zz / 5\Zz$ at distance at least~2 modulo~5 from both~$i$ and~$j$.

  If there exists $z \in X$ with index~$k$ modulo~5 such that~$x \qle z$ and~$z \qle y$ (which can again be tested with~$\Phi$),
  then~$x \qle y$ by transitivity, and similarly if $y \qle z \qle x$.
  Let us thus assume that any $z \in X$ with index~$k$ modulo~5
  is either before both~$x$ and~$y$ in~$\qle$, or after both of them.

  For any~$t$, pick~$z_t \in X$ with index~$5t+k$ (as long as it is in the interval~$[n]$),
  which exists by assumption.
  Then there must be some~$s$ such that $z_t \qle x,y$ when~$t \le s$ and~$z_t \qge x,y$ when~$t > s$
  (with possibly~$s = 0$ or~$s = n$).
  Thus~$i$ and $j$ are between $5s+k$ and $5(s+1)+k$, and are equal to neither of them since~$i,j$ are not~$k$ modulo~5.
  It follows that that the ordering of~$i,j$ only depends on their values modulo~$5$, which we have guessed.
\end{proof}

Let us call~$\rho$ the recolouring function of the bags~$B_i$.
It is the same for all bags since it is part of the abstraction~$\alpha$,
and it is idempotent ($\rho \circ \rho = \rho$).
Consider any vertex of~$B$, say~$x \in V(B_i)$.
In the bag~$B_i$, $x$ is given some colour~$\lambda_{B_i}(x)$,
which we call its \emph{initial colour}~$\colinit(x)$.
Then, as we multiply by~$B_{i+1}$, the colour of~$x$ changes to~$\rho(\colinit(x))$.
Continuing with~$B_{i+2}$ will again recolour~$x$ by applying~$\rho$,
but since~$\rho$ is idempotent, this does nothing.
Thus, for any~$j>i$, the colour of~$x$ in~$B_1\cdot \ldots \cdot B_j$ is~$\rho(\colinit(x))$.
We call this its \emph{final colour} and denote it by~$\colfin(x)$.

Recall that to define the abstraction~$\alpha$ of the bags~$B_i$,
we first defined the vertex type $\type_{B_i}(x)$ of~$x$ inside the bag~$B_i$,
which indicates its colour~$\lambda_{B_i}(x)$, \ie the initial colour,
and the direction of edges between~$x$ and input vertices in~$B_i$.
In what follows, the type of a vertex~$x$ will always be understood 
relative to the bag~$B_i$ containing it, and we shorten it to~$\type(x)$.
Using a non-deterministic colouring, we can guess the types of all vertices in $B$,
and we assume this colouring to be given in all subsequent lemmas.

\begin{lemma}
 \label{lem:edge-direction-type}
 Given vertices~$x,y$ with~$\idx(x)+2 \le \idx(y)$,
 the direction of the edge between $x$ and $y$ only depends on~$\type(x),\type(y)$, and~$\rho$.
\end{lemma}

\begin{proof}
  Say~$\idx(x) = i$ and~$\idx(y) = j$.
  By definition of the product of bags, the direction of the edge between $x$ and $y$
  is the same as that of the edge between $c$ and $y$ in~$B_j$, where~$c$ is the colour of~$x$ in~$B_i \cdot \ldots \cdot B_{j-1}$.
  Since~$j \ge i+2$, the colour~$c$ is $\colfin(x)$, which only depends on~$\type(x)$ and~$\rho$.
  The direction of the edge between $c$ and $y$ in~$B_j$ is then given by~$\type(y)$.
\end{proof}

With this in hand, we can already order the vertices of a given type~$\tau$.
Denote by~$V_\tau = \bigcup_i \tau(B_i)$ the set of vertices of type~$\tau$.
Note that if~$V_\tau$ is non-empty, then it intersects all bags~$B_1,\dots,B_n$.
Indeed whether or not~$B_i$ contains a vertex of type~$\tau$ is indicated by the abstraction~$\alpha$.

\begin{lemma}
 \label{lem:ordering-type}
 For any type~$\tau$, there is an \fo transduction depending only on~$k,\rho,\tau$ that orders~$V_\tau$ according to~$\qle$.
\end{lemma}

\begin{proof}
  By \cref{lem:ordering-approx}, it is sufficient to approximately order~$V_\tau$.
  According to \cref{lem:edge-direction-type}, for any~$x,y$ of type~$\tau$ satisfying~$\idx(x)+2 \le \idx(y)$,
  either there always is the edge~$x \to y$, or always~$x \from y$,
  and this direction depends only on~$\tau,\rho$.
  Either way, one can order~$x,y$ with indices differing by at least~2
  simply by looking at the direction of the edge between them.
\end{proof}

\newcommand{\syn}{\mathsf{syn}}
Thus we can order the vertices of each type independently, and the issue is now to combine these orderings.
Denote by~$\Gamma$ the set of inhabited vertex types in~$B_1,\dots,B_n$.
Recall that subsets of vertices~$X,Y$ are called \emph{homogeneous}
if either all edges are oriented from~$X$ to~$Y$, or all from~$Y$ to~$X$.
We define the \emph{synchronisation graph}~$G^\syn$ with vertices~$\Gamma$,
in which~$\sigma,\tau$ are adjacent if and only if~$V_\sigma$ and~$V_\tau$ are \emph{not} homogeneous.

We will show that when~$G^\syn$ is connected, one can transduce the quasi-ordering~$\qle$
corresponding to the original decomposition~$B_1\cdot \ldots \cdot B_n$.
Then, we will handle the case where~$G^\syn$ is disconnected by slightly modifying this decomposition.
We start with only two types that are non-homogeneous.
\begin{lemma}
  \label{lem:ordering-types-heterogeneous}
  Let~$\sigma,\tau \in \Gamma$ be types such that~$V_\sigma,V_\tau$ are non-homogeneous.
  Then~$V_\sigma \cup V_\tau$ can be ordered according to~$\qle$ by an \fo transduction depending only on~$k,\rho,\sigma,\tau$.
\end{lemma}
\begin{proof}
  Consider~$x \in V_\sigma$ and~$y \in V_{\tau}$ which we are trying to order.
  Using \cref{lem:ordering-approx}, we can assume that their indices differ by at least~2.
  We first consider the orientation of the edge between $x$ and $y$.
  By \cref{lem:edge-direction-type}, it only depends on~$\rho$, the types~$\sigma,\tau$,
  and on which of~$x$ and $y$ has the smaller index.
  It is simple to check that this leaves only four possibilities:
  \begin{itemize}
    \item Either the edge is always oriented from the smaller index to the larger one,
      or symmetrically from the larger to the smaller.
      These are the simple cases: as in \cref{lem:ordering-type},
      we immediately obtain the ordering of~$x$ and $y$ from the edge between them.
    \item Otherwise, the edge is always oriented from~$x$ to~$y$, \ie
      from the type~$\sigma$ to the type~$\tau$, or vice versa.
  \end{itemize}

  Let us thus assume that we are in the second case, and
  without loss of generality, for~$\type(x) = \sigma$ and~$\type(y) = \tau$,
  the edge is always oriented as $x \to y$, unless their indices differ by less than~2.
  On the other hand, since~$V_\sigma$ and~$V_{\tau}$ are not homogeneous,
  there must exist an edge~$u \from v$ with~$\type(u) = \sigma$ and~$\type(v) = \tau$.
  We call such an edge from~$V_\tau$ to~$V_\sigma$ a \emph{backward edge}.

  Call~$V_{i,\sigma} = \sigma(B_i)$ the set of vertices with type~$\sigma$ in~$B_i$.
  Recall that the abstraction of~$B_i$ indicates whether or not~$V_{i,\sigma}$ and~$V_{i,\tau}$ are homogenous,
  and, when they are, the direction of edges between them.
  Thus this information does not depend on~$i$, since all~$B_i$ have the same abstraction.
  Similarly, using the definition of types and product of bags, one may check that
  the direction of edges between~$V_{i,\sigma}$ and~$V_{i+1,\tau}$ (which are always homogeneous)
  does not depend on~$i$, and idem with~$V_{i-1,\tau}$.

  It follows that if the backward edge~$u \from v$ is for instance
  between~$u \in V_{k,\sigma}$ and~$v \in V_{k+1,\tau}$,
  then there is also a backward edge between~$V_{i,\sigma}$ and~$V_{i+1,\tau}$ for all~$i$.
  In all cases, we obtain that for all~$i$, there is a backward edge between~$V_{i,\sigma}$
  and one of $V_{i-1,\tau},V_{i,\tau},V_{i+1,\tau}$.

  We are now able to order~$x$ and $y$ as follows.  Say that~$x \in V_{i,\sigma}$.
  Using \cref{lem:ordering-type}, we can identify the set~$V_{i,\sigma}$.
  Then, we can find some backward edge~$u' \from v'$ with~$u' \in V_{i,\sigma}$ and~$v'$ of type~$\tau$.
  By the previous arguments, such a backward edge must exist, and necessarily $\idx(v') \in \{i-1,i,i+1\}$.
  Finally, we compare~$v'$ and~$y$ which are both of type~$\tau$ using \cref{lem:ordering-type}.
  Since the index of~$y$ is by assumption less than~$i-1$ or more than~$i+1$,
  it does not matter that the index of~$v'$ is only known up to plus or minus~1.
\end{proof}

Once we can order vertices of two non-homogeneous types, it is simple to propagate the ordering to a connected component of~$G^\syn$.

\begin{lemma}\label{lem:ordering-types-sync}
  Let~$\Theta \subset \Gamma$ be a connected component of the synchronisation graph~$G$,
  and denote~$V_\Theta = \bigcup_{\tau \in \Theta} V_\tau$.
  Then~$V_\Theta$ can be ordered according to~$\le$ by an \fo transduction depending only on~$k,\alpha,\Theta$.
\end{lemma}

\begin{proof}
Say, without loss of generality, we want to compare vertices $x$ and $y$ of types $\sigma$ and $\tau$, respectively, where $\sigma$ and $\tau$ belong to the same connected component $\Theta$. As $\sigma,\tau \in \Theta$, there exists a path
\[ P: \sigma = \gamma_0, \gamma_1, \dots, \gamma_h = \tau \]
connecting them in the synchronisation graph.

For each pair~$\gamma_i,\gamma_{i+1}$, \cref{lem:ordering-types-heterogeneous} gives a transduction that orders~$V_{\gamma_i} \cup V_{\gamma_{i+1}}$.
Combining all these transductions, we may assume we have access to a relation~$\qle_i$
which coincides with~$\qle$ on any pair~$x,y$ with~$x \in V_{\gamma_i}$ and $y \in V_{\gamma_{i+1}}$, or vice versa.
Then we can also define $x \sim_i y$ as $x \qle_i y \land y \qle_i x$ to check that~$x,y$ are in the same bag under the same assumption on types.
Recall also that we assume to have already guessed the type of each vertex with colouring,
and let us write~$\gamma_i(x)$ for the predicate checking that~$x$ has type~$\gamma_i$.

Now we can order $x \in V_\sigma$ and $y \in V_\tau$ by the formula:
\begin{align*}
  \exists x_1,\dots,x_{h-1},\quad &
  \gamma_1(x_1) \land \dots \land \gamma_{h-1}(x_{h-1}) \\
  \land\ & x \sim_0 x_1 \land x_1 \sim_1 x_2 \land \dots \land x_{h-2} \sim_{h-2} x_{h-1} \\
  \land\ & x_{h-1} \qle_{h-1} y.
\end{align*}
Indeed, since each type~$\gamma_i$ is inhabited, and thus intersects every type,
there must be some~$x_i$ of type~$\gamma_i$ in the same bag as~$x$.
The first line of the formula checks that~$x_1,\dots,x_{h-1}$ have the correct types,
the second line checks that they are in the same bag as~$x$, and the last line compares~$x_{h-1}$ and~$y$,
which is equivalent to comparing~$x$ and~$y$.
\end{proof}

When~$G^\syn$ is connected, \cref{lem:ordering-types-sync} concludes the proof of \cref{lem:idempotent-partition}.
Note in that case that the given factorisation $B = B_1\cdot  \ldots \cdot B_n$ is not modified,
\ie the partition $\{X_1,\dots,X_m\}$ in the conclusion of \cref{lem:idempotent-partition} is exactly $\{B_1,\dots,B_n\}$,
and conditions~\eqref{item:fact-depth} and~\eqref{item:part-rank} of the statement are trivial.

Assume finally that~$G^\syn$ has several connected components that we enumerate as $\Theta_1,\dots,\Theta_r$.
We then reorganise the decomposition of~$B$.
For $i \in [n]$ and $t \in [r]$, call~$X_{t,i}$ the set of internal vertices of~$B_i$ with types in~$\Theta_t$.
By \cref{lem:ordering-types-sync}, for each $t \in [r]$, the set $V_{\Theta_t}$ can be ordered according to~$\qle$ with a transduction.
Let $\qle'$ be the quasi-ordering on the partition
$\{X_{t,i}\mid i \in [n], t \in [r]\}$ obtained by ordering first by type and then by bag, i.e.\
\[ X_{1,1} \qle' \dots \qle' X_{1,n} \qle' \dots \qle' X_{r,1} \qle' \dots \qle' X_{r,n}. \]

This quasi-ordering $\qle'$ has small cut-rank:
\begin{lemma}\label{lem:small-rank}
  For any $x \in V(T)$,
  \[ \rk_T(\{y : y \qle' x\}); \{y : y \qgt' x\}) \le k \cdot (2^k+1). \]
\end{lemma}
\begin{proof}
  Any bipartition as in the statement can be written as $(X,Y)$
  with $X=\big(\bigcup_{t'<t} V_{\Theta_{t'}}\cup \bigcup_{j\leq i} X_{t,j}\big)$
  and $Y=\big(\bigcup_{t'>t} V_{\Theta_{t'}}\cup \bigcup_{j>i} X_{t,j}\big)$
  for some $t \in [r]$ and $i \in [n]$.
  Recall that the total number of types (for bags of order~$k$) is bounded by $k \cdot 2^k$.
  By definition of the synchronisation graph, between $\bigcup_{t'<t} V_{\Theta_{t'}}$ and $\bigcup_{t' \ge t} V_{\Theta_{t'}}$,
  the direction of edges only depend on the types of vertices.
  This implies
  \begin{align}
    \rk_T\left(\bigcup_{t'< t} V_{\Theta_{t'}}; \big(\bigcup_{t'> t} V_{\Theta_{t'}}\cup \bigcup_{j> i} X_{t,j}\big)\right)
    & \leq \frac{k\cdot 2^k}{2}, & \text{and similarly}\\
    \rk_T\left(\bigcup_{j\leq i} X_{t,j}; \bigcup_{t'> t} V_{\Theta_{t'}}\right) & \leq \frac{k\cdot 2^k}{2}.
  \end{align}
  On the other hand, since the bags have order~$k$, any prefix/suffix bipartition in the ordering of bags has rank at most~$k$
  (see \cref{lem:cw-rw,lem:monoid-nlcwidth}). Thus
  \begin{equation}
    \rk_T\left(\bigcup_{j\leq i} X_{t,j}; \bigcup_{j> i} X_{t,j}\right) \leq k.
  \end{equation}
  The three inequalities directly combine to prove $\rk(X;Y) \le k \cdot (2^k+1)$.
\end{proof}

We can now complete the proof of \cref{lem:idempotent-partition}, by checking the three conditions of the statement.
\begin{enumerate}
  \item Firstly, each~$T[X_{t,i}]$ is the internal tournament of a bag in $\mathfrak{C}_k$ of factorisation depth at most~$(p-1,q+2p)$.
    Indeed, this follows from \cref{lem:type-deletion} since $T[X_{t,i}]$ is the restriction of~$T[V(B_i)]$ to some union of types,
    and~$B_i$ is a bag in $\mathfrak{C}_k$ and has factorisation depth at most~$(p-1,q)$ by assumption.
  \item By \cref{lem:small-rank}, the quasi-ordering
    \[ X_{1,1} \qle' \dots \qle' X_{1,n} \qle' \dots \qle' X_{r,1} \qle' \dots \qle' X_{r,n}. \]
    has cut-rank at most~$k \cdot (2^k+1)$.
  \item Finally, there is an extension \fo transduction~$\Psi_k$ depending only on~$k$ such that
    for $T = (V,E)$ the internal tournament of~$B$, we have $(V,E,\qle') \in \Psi_k(T)$.
    Indeed, from \cref{lem:ordering-types-sync}, we can obtain the restriction of~$\qle'$ to any~$V_{\Theta_t}$ by a transduction depending only on~$k,\alpha,\Theta$.
    By first guessing the type of each vertex, ordering by types between the different sets $V_{\Theta_1}, \dots, V_{\Theta_r}$,
    and by the previous transductions inside each~$V_{\Theta_t}$, we obtain the desired quasi-ordering~$\qle'$.

    This transduction depends on~$k$, the abstraction~$\alpha$ of bags, and the synchronisation graph~$G^\syn$.
    But, as argued at the start of \cref{subsec:main-lemma}, there are only finitely many possible choices for~$\alpha$ and for~$G^\syn$ when~$k$ is fixed,
    thus the transduction~$\Psi_k$ may first guess~$\alpha$ and~$G^\syn$, and then proceed as above.
\end{enumerate}

This completes the proof of \cref{lem:idempotent-partition}, and with it \cref{thm:transduce-ordering}.

\section{Definable decompositions and definable properties}\label{sec:definable-decomp}
\Cref{thm:transduce-ordering} proved that in tournaments of bounded linear clique-width, one can \fo-transduce orderings of bounded cut-rank.
We now show in \cref{sec:transduce-decomp} that from this ordering, one can also transduce a linear clique decomposition, proving our main result, \cref{thm:main}.
Then, in \cref{sec:emso}, we combine this with some of the central ideas from Courcelle's work~\cite{CourcelleE2012}
to obtain that the logics \mso and \emso are equivalent over bounded linear clique-width tournaments (\cref{thm:emso}).

\subsection{FO-transducing linear decompositions}\label{sec:transduce-decomp}
Let us first explain how we represent a linear decomposition as a relational structure.
We choose a representation based on the monoid~$\mathfrak{C}_k$ from \cref{Section: Linear decompositions and semigroups},
but it would be simple to obtain the same results with a representation based on the operations defining linear clique-width or linear NLC-width.

We use the following standard encoding of words as relational structures:
for a finite alphabet~$A$, a word~$w_1,\dots,w_n$ in~$A^*$ is represented by the ordered structure with
vertex set~$\{1,\dots,n\}$, one binary relation~$<$ for the natural ordering on~$\{1,\dots,n\}$,
and for each letter $a \in A$, a unary relation~$a(x)$ such that~$a(i)$ holds if and only if $w_i = a$.
We call this the \emph{unary representation} of the word~$w_1,\dots,w_n$.

Now for any~$k \in \Nn$, let $\mathfrak{A}_k$ denote the set of atomic bags of order~$k$ (up to isomorphism).
Recall that~$\mathfrak{A}_k$ by definition generates the monoid~$\mathfrak{C}_k$,
and that linear NLC decompositions are equivalent, in the sense of \cref{lem:monoid-nlcwidth}
to words $B_1,\dots,B_n$ with each $B_i \in \mathfrak{A}_k$.
We then represent the linear decomposition~$B_1,\dots,B_n$ as a word over the alphabet~$\mathfrak{A}_k$, in unary representation.

Firstly, from a linear decomposition, one can reconstruct the corresponding graph with a transduction.
\begin{lemma}\label{lem:decomp-to-graph}
 For any~$k$, one can compute an \fo transduction~$\Phi_k$ which given a linear decomposition $B_1,\dots,B_n$ with $B_i \in \mathfrak{A}_k$, outputs the underlying graph of $B_1 \cdot \ldots \cdot B_n$.
\end{lemma}
\begin{proof}
  If~$\Phi_k$ is allowed to be an \mso interpretation instead of an \fo transduction,
  then \cref{lem:decomp-to-graph} is an elementary result, see \cite[Proposition~7.30]{CourcelleE2012}.
  Furthermore, Colcombet \cite[Corollary~1]{colcombet2007deterministic} proved that
  any \mso interpretation defined on labelled trees (encoded with the ancestor relation) is equivalent to an \fo transduction.
  Unary representations of words---and thus in particular linear decompositions---are a special case of trees,
  where the left-to-right ordering corresponds to the ancestor relation.
  Thus Colcombet's theorem proves that the map~$\Phi_k$ can also be obtained as \fo interpretation.
\end{proof}

\transducedecomp*
\begin{proof}
  Assuming~$T$ is a tournament with linear clique-width at most~$k$ (and thus also linear NLC-width at most $k$),
  apply \cref{thm:transduce-ordering} to obtain an ordering~$<$ of~$V(T)$ with cut-rank at most some~$k''$ function of~$k$.
  Enumerate the vertices of~$T$ as $v_1 < \dots < v_n$.
  By \cref{lem:cw-rw,lem:monoid-nlcwidth}, there is a linear decomposition $B_1,\dots,B_m$ of~$G$,
  with $B_i \in \mathfrak{A}_{k'}$ for $k' \eqdef 2^{k''}+1$, and whose vertex introduction ordering is~$<$.
  Without loss of generality, each~$B_i$ has (exactly) one internal vertex.
  Indeed, if~$B_i$ has no internal vertex, it can simply be combined with~$B_{i+1}$ or~$B_{i-1}$.
  Thus $m=n$, and the internal vertex of~$B_i$ is~$v_i$.

  Now we already have vertices $\{v_1,\dots,v_n\}$, and the ordering~$<$ from \cref{thm:transduce-ordering},
  hence~$\Phi_k(T)$ can obtain the unary encoding of $B_1,\dots,B_n$ by guessing the unary relations encoding this word with a non-deterministic colouring.
  It only remains to check that this decomposition is valid, i.e.\ that the previous guess was correct.
  To this end, apply \cref{lem:decomp-to-graph} to obtain the graph corresponding to the decomposition $B_1,\dots,B_n$, and check that this graph coincides with~$T$ given as input.
  This last check ensures that condition~(2) from the statement holds.
\end{proof}

\subsection{Application to MSO-definable properties}\label{sec:emso}
Applications of transductions between graphs and tree- or clique-decompositions in Courcelle's work
typically involve the \emph{Backwards Translation Theorem} \cite[Theorem~1.40]{CourcelleE2012},
which states that for any \mso transduction~$\Phi$ with signature $\Gamma \to \Delta$,
and any \mso formula~$\psi$ on~$\Delta$, there is an \mso formula $\psi \circ \Phi$ such that
\[ G \models \psi \circ \Phi \quad \iff \quad \exists H \in \Phi(G),\ H \models \psi. \]
We will use a similar statement to prove \cref{thm:emso} from \cref{thm:main}.

Backwards translation is easily adapted to \cmso formulas and transductions,
as well as \fo formulas and \emph{deterministic} \fo transductions.
On the other hand, \fo formulas are not sufficiently expressive to handle the non-deterministic colouring operation.
To this end, one should instead use \emph{existential \mso} (\emso) formulas, that is formulas of the form
\[ \exists X_1,\dots,X_n,\ \psi, \]
where the variables~$X_i$ range over subsets of vertices,
and~$\psi$ is an \fo formula that may use the predicates $x \in X_i$.
The proof remains essentially the same as in the \mso case.
\begin{lemma}[Backwards Translation for \fo Transductions]\label{lem:backward-translation-fo}
 For any \fo transduction~$\Phi$ with signature $\Gamma \to \Delta$,
 and any \emso formula~$\psi$ on~$\Delta$, there is an \emso formula $\psi \circ \Phi$ such that
 \[ G \models \psi \circ \Phi \quad \iff \quad \exists H \in \Phi(G),\ H \models \psi. \]
\end{lemma}

\begin{proof}[Sketch of proof]
  It suffices to prove the result when taking~$\Phi$ to be each of the five operations constituting \fo transductions independently.
  \begin{description}
    \item[interpretation] If~$\Phi$ is an interpretation $\Gamma \to \Delta$,
      define $\psi \circ \Phi$ by replacing each occurrence of a relation $R(x_1,\dots,x_r)$ in~$\psi$ for $R \in \Delta$
      by the corresponding formula $\phi_R(x_1,\dots,x_r)$ from~$\Phi$.
    \item[colouring] If~$\Phi$ is a non-deterministic colouring adding unary predicates $C_1,\dots,C_c$,
      then define $\psi \circ \Phi$ as
      \[ \exists X_1,\dots,X_n,\ \psi, \]
      after replacing $C_i(x)$ with $x \in X_i$ inside~$\psi$.
    \item[filtering] If~$\Phi$ is a filtering that checks that the sentence~$\phi$ holds,
      then define~$\psi \circ \Phi$ as~$\psi \land \phi$.
    \item[universe restriction] If~$\Phi$ restricts the vertex set to those~$x$ satisfying~$\phi(x)$,
      then define $\psi \circ \Phi$ by replacing first-order quantifiers in~$\psi$ as follows:
      \begin{align*}
        \exists x.\ \theta(x) \quad & \text{becomes} \quad \exists x.\ \phi(x) \land \theta(x) & \text{and} \\
        \forall x.\ \theta(x) \quad & \text{becomes} \quad \forall x.\ \phi(x) \rightarrow \theta(x)
      \end{align*}
      The leading \emso quantifiers $\exists X_1,\dots,X_n$ need not be modified.
    \item[copying] When~$\Phi$ copies each vertex~$v$ into $(v,1),\dots,(v,k)$,
      defining $\psi \circ \Phi$ involves simulating a quantifier $\exists (v,i)$ over copies of vertices
      by a quantification $\exists v$, and a case disjunction over~$i$.
      As this case is far more tedious than the others, and none of our transductions use copying, we do not detail it further.
      \qedhere
  \end{description}
\end{proof}

The last tool we need to prove \cref{thm:emso} is the following well-known variant of Büchi--Elgot--Trakhtenbrot theorem.
\begin{theorem}[Büchi--Elgot--Trakhtenbrot \cite{Buchi1960,Elgot1961,Trakhtenbrot1962}]\label{thm:BET}
 For any finite alphabet~$A$ and language $L \subseteq A^*$, the following are equivalent:
 \begin{enumerate}
   \item $L$ is a regular language,
   \item $L$ is defined by an \mso sentence,
   \item $L$ is defined by an \cmso sentence,
   \item $L$ is defined by an \emso sentence.
 \end{enumerate}
\end{theorem}

\emsoequiv*
\begin{proof}
  Take the transduction~$\Phi_k$ and $k' \in \Nn$ given by \cref{thm:main}.
  Applying backwards translation for \cmso to the transduction from linear decompositions to graphs given by \cref{lem:decomp-to-graph},
  we obtain a \cmso formula~$\phi'$ on linear decompositions of width at most~$k'$,
  such that for any linear decomposition~$D$ with width at most~$k'$ of a tournament~$T$,
  we have $D \models \phi'$ if and only if $T \models \phi$.
  Since linear decompositions are words, \cref{thm:BET} gives that~$\phi'$ is equivalent to some \emso formula~$\psi'$.

  Now apply backwards translation for \fo (\cref{lem:backward-translation-fo}) to~$\psi'$ and the transduction~$\Phi_k$ given by \cref{thm:main}.
  This yields an \emso formula~$\psi$ such that $T \models \psi$
  if and only if $D \models \psi'$ holds for some $D \in \Phi_k(T)$.

  Consider~$T$ a tournament with linear NLC-width at most~$k$.
  Assume that~$T \models \phi$, and consider a linear decomposition $D \in \Phi_k(T)$ of~$T$ with width at most~$k$, guaranteed by \cref{thm:main}.
  Then the choice of~$\phi'$, $\psi'$, and~$\psi$ directly gives
  \[ T \models \phi 
    \quad \iff \quad D \models \phi'
    \quad \iff \quad D \models \psi'
    \quad \implies \quad T \models \psi. \]
  Conversely, assume that $T \models \psi$, hence there is some $D \in \Phi_k(T)$ such that $D \models \psi'$.
  Then \cref{thm:main} guarantees that~$D$ is a linear decomposition of~$T$ with width at most~$k$,
  and we once again have the equivalence between $D \models \psi'$, $D \models \phi'$, and finally $T \models \phi$.
  Therefore~$\phi$ is equivalent to the \emso sentence~$\psi$ for any tournament~$T$ with linear NLC-width at most~$k$.
\end{proof}

\bibliographystyle{plainurl}
\bibliography{biblio}

\end{document}